\documentclass[11pt]{amsart}
\usepackage{amsmath, amsfonts,amssymb}
\usepackage[usenames]{color} 
\usepackage[all]{xy}     
\usepackage[ top=1in, bottom=1in, left=1.1in, right=1.10in ]{geometry}
\usepackage[colorlinks=true, linkcolor=blue, citecolor=green, urlcolor=cyan, pagebackref, linktocpage=true]{hyperref}
\usepackage{stmaryrd}
\usepackage{pst-plot}
\usepackage{pst-math}
\usepackage{pst-func}
\usepackage{palatino}

\usepackage{color}

\newcommand{\cK}{\mathcal{K}}

\newcommand{\cE}{\mathcal{E}}
\newcommand{\cF}{\mathcal{F}}
\newcommand{\cO}{\mathcal{O}}

\newcommand{\NN}{\mathbb N}  
\newcommand{\ZZ}{\mathbb Z}

\newcommand{\QQ}{\mathbb Q}
\newcommand{\CC}{\mathbb{C}}

 \newcommand{\fm}{\mathfrak{m}}
 \newcommand{\fn}{\mathfrak{n}}
 \newcommand{\fp}{\mathfrak{p}}

\newcommand{\inj}{\hookrightarrow}

\newcommand{\Ext}{\operatorname{Ext}}
\newcommand{\Ass}{\operatorname{Ass}}
\newcommand{\Hom}{\operatorname{Hom}}
\newcommand{\Dim}{\operatorname{dim}}

\newcommand{\End}{\operatorname{End}}
\newcommand{\Ker}{\operatorname{Ker}}
\newcommand{\IM}{\operatorname{Im}}
\newcommand{\Ann}{\operatorname{Ann}}
\newcommand{\InjDim}{\operatorname{inj.dim}}
\newcommand{\Supp}{\operatorname{Supp}}

\newcommand{\Cech}{ \check{\rm{C}}}

\newcommand{\Spec}{\operatorname{Spec}}
\newcommand{\Proj}{\operatorname{Proj}}

\newcommand{\Length}{\operatorname{length}}

\newcommand{\height}{\operatorname{ht}}

\newcommand{\link}{\operatorname{link}}

\DeclareMathOperator{\dualizing}{\omega^{\bullet}_X}
\newcommand{\D}{D}
\newcommand{\E}{E}

\newcommand{\FDer}[1]{\stackrel{#1}{\longrightarrow}}

\newcommand{\surj}{\twoheadrightarrow}

\newcommand{\itLC}[3]{H^{{#1}_{#2}}_{I_{#2}} \cdots H^{{#1}_2}_{I_2} H^{{#1}_1}_{I_1}\left(#3\right)}
\newcommand{\Mult}{e} 
\newcommand{\LyuMixed}[3]{\widetilde{\lambda}_{#1, #2}\left( #3 \right) }
\newcommand{\connarrow}{\ar `[d] `[l] `[llld] `[ld] [lld]}

\renewcommand{\(}{\left(}
\renewcommand{\)}{\right)}

\newcommand{\edit}[1]{#1}

\newtheorem{theorem}{Theorem}[section]

\newtheorem{proposition}[theorem]{Proposition}
\theoremstyle{definition}
\newtheorem{definition}[theorem]{Definition}
\newtheorem{properties}[theorem]{Properties}
\newtheorem{ThmDef}[theorem]{Theorem/Definition}
\newtheorem{example}[theorem]{Example}

\theoremstyle{remark}
\newtheorem{remark}[theorem]{Remark}
\newtheorem{question}[theorem]{Question}
\newtheorem{notation}[theorem]{Notation}
\numberwithin{equation}{section}

\begin{document}
\title{A survey on the Lyubeznik numbers}
\author{Luis N\'u\~nez-Betancourt}  
\address{Department of Mathematics, University of Virginia, Charlottesville, VA 22904, USA.}
\email{lcn8m@virginia.edu}

\author{Emily E.\ Witt}  
\address{Department of Mathematics, University of Minnesota,  Minneapolis, MN 55455, USA.}
\email{ewitt@umn.edu}

\author{Wenliang Zhang}  
\address{Department of Mathematics, University of Nebraska,  Lincoln NE 68588, USA.}
\email{wzhang15@unl.edu}


\thanks{The third author was partially supported by the NSF grant DMS \#1247354/\#1068946.}
\maketitle


\begin{abstract}
The Lyubeznik numbers are invariants of a local ring containing a field that capture ring-theoretic properties, but also have numerous connections to geometry and topology.  
We discuss basic properties of these integer-valued invariants, as well as describe some significant results and recent developments (including certain generalizations) in the area.
\end{abstract}

\setcounter{tocdepth}{1}
\tableofcontents

\section{Introduction}
Since the introduction of the Lyubeznik numbers \cite{LyubeznikFinitenessLocalCohomologyModules},  
the study of these invariants of local rings containing a field has grown in several compelling directions. 
One aim of this paper is to present definitions used in the study of the Lyubeznik numbers, along with  
examples and applications of these invariants, to those new to them.
For those familiar with them, we also present recent results on, and generalizations of, the Lyubeznik numbers, 
as well as open problems in the area.

Given a module $M$ over a ring $S$, if $I^\bullet$ is a minimal injective resolution of $M$,
then each $I^i$ is isomorphic to a direct sum of indecomposable injective modules
 $E_S(S/\fp)$, $\fp$ is a prime ideal of $S$;
\emph{i.e.}, injective hulls over $S$ of $S/\fp$.
The number of copies of  $\E_S(S/\fp)$ in $I^i$ is the \emph{$i$-th Bass number of $M$ with respect to $\fp$}, 
denoted $\mu_i(\fp, M)$ and equals $\dim_{S_\fp/\fp S_\fp} \Ext^i_S(S_\fp/\fp S_\fp, M_\fp).$

Huneke and Sharp proved that if $S$ is a regular ring of characteristic $p>0$ and $I$ is an ideal of $S$, then the Bass numbers of the local cohomology modules of the form $H^j_I(S)$, $j\in\NN$, are finite, raising the analogous question in the characteristic zero case  \cite{Huneke-Sharp}.  
Utilizing $D$-module theory, Lyubeznik  proved the same statement for regular local rings of characteristic zero containing a field \cite{LyubeznikFinitenessLocalCohomologyModules}.

Relying on the finiteness of the Bass numbers of local cohomology modules, Lyubeznik introduced a family of integer-valued invariants associated to a local ring containing a field, now called \emph{Lyubeznik numbers}.  They are defined as follows:  
Suppose that  $(R,\fm,K)$ is a local ring admitting a surjection from an $n$-dimensional regular local ring $(S,\mathfrak{n},K)$ containing a field, and let $I$ denote the kernel of the surjection.  Given $i, j \in \NN$, the \emph{Lyubeznik number of $R$ with respect to $i,j \in \NN$}, 
is defined as $\Dim_K \Ext^i_S\left(K,H^{n-j}_I (S)\right)$, and is denoted $\lambda_{i,j}(R)$.  Notably, this invariant depends only on $R, i$, and $j$;
in particular, it is independent of the choice of $S$ and of the surjection.  Moreover, if $R$ is any local ring containing a field, then defining $\lambda_{i,j}(R)$ as $\lambda_{i,j}(\widehat{R})$, \edit{where $\widehat{R}$ denotes the completion of $R$ at $\fm$}, extends the definition.

Lyubeznik numbers are indicators of certain ring-theoretic properties.  For example, let $R$ be a local ring containing a field, and let $d=\dim(R).$ 
Then  $\lambda_{d,d}(R)$ equals the number of connected components of the Hochster-Huneke graph of the completion of the strict 
Henselization of $R$ \cite{LyubeznikSomeLocalCohomologyInvariants, ZhanghighestLyubeznikNumber}.  Consequently, 
\begin{enumerate} 
\item $\lambda_{d,d}(R)\neq 0$ for $d=\dim(R)$ \cite{LyubeznikFinitenessLocalCohomologyModules};
\item $\lambda_{d,d}(R)=1$ if $R$ is Cohen-Macaulay \cite{KawasakiHighestLyubeznikNumber},  \edit{or even if $R$ satisfies Serre's condition $S_2$ ({\it cf.} Theorem \ref{S2 implies lambda to be 1})}; and
\item $\lambda_{d,d}(R)=1$ if $R$ is analytically normal \cite{LyubeznikFinitenessLocalCohomologyModules}. 
\end{enumerate}

Strikingly, the Lyubeznik numbers, which are defined in a purely algebraic context, have extensive connections with geometry and topology. 
For example, if $R$ is the local ring of an isolated singularity of a complex space of pure dimension at least two, 
then every Lyubeznik number of $R$ equals the $\CC$-vector space dimension of a certain singular cohomology group \cite{GarciaSabbah}.
Blickle and Bondu give similar connections between Lyubeznik numbers and \'etale cohomology in the positive characteristic setting \cite{B-B}, \edit{which was generalized by Blickle in \cite{BlickleLyuCIS}.} 

The Lyubeznik numbers of a Stanley-Reisner ring 
capture properties of the associated simplicial and combinatorial structure \cite{AM-CharCyclesI,AM-CharCyclesII,AM-GL-Z,MustataSRrings,AMV,Yanagawa}. 
For instance, the Lyubeznik numbers 
are a measure of the acyclicity of 
certain linear strands of the Alexander dual of the 
simplicial complex associated to the ring.  
In addition, there are several algorithms for computing the Lyubeznik numbers of these rings
\cite{AM-CharCyclesI,AMV}.

{
For rings of characteristic zero, there are algorithms for computing Lyubeznik numbers 
that employ the $D$-module structure of local cohomology modules \cite{WaltherAlgorithm, AlvarezMontanerLeykin}.
In \cite{AM-CharCyclesI,AlvarezMontanerLeykin}, these numbers are computed as multiplicities of the characteristic cycles associated to certain 
local cohomology modules.
}

Due to their wide range of applications, several variants and generalizations of the Lyubeznik numbers have been created.
\edit{Analogous invariants of a simplicial normal Gorenstein semigroup ring modulo a squarefree monomial ideal are proven well defined in \cite{Yanagawa}.}
There is a version for projective varieties over a field of prime characteristic, and it is not known whether the definition can be extended to the equal characteristic zero case \cite{ZhangLyubeznikNumbersProjSchemes}.
In characteristic zero, there is a family of invariants closely related to the Lyubeznik numbers defined in terms of certain characteristic cycle multiplicities from $D$-module theory \cite{AM-NumInv}.
For all local rings containing a field, there is a generalization of the Lyubeznik numbers defined using $D$-modules, with several ideals and a coefficient field as parameters; these invariants can be used to measure singularities in positive characteristic \cite{BlickleIntHomology,NuWiEqual}.
Recently, an alternate version of the Lyubeznik numbers has been defined for, in particular, local rings of mixed characteristic \cite{NunezBWittLyuNumMixed}.

\section{Preliminaries}
\subsection{Local cohomology}\label{PreLocalCohomology}
Before defining the Lyubeznik numbers, we need to introduce \emph{local cohomology modules}.  
We include some results on this topic that will be used as tools later in this survey.  
Please see, among many beautiful references, \cite{BrodmannSharpLocalCohomology,24HoursLocalCohomology,LyuSurveyLC} for more details. 

Let $S$ be a Noetherian ring, and fix $f_1,\ldots,f_\ell\in S$. Consider the $\check{\mbox{C}}$ech-like complex, $\Cech^\bullet(\underline{f};S)$:
$$
0\to S\to \bigoplus_{1 \leq i \leq \ell} S_{f_i}\to\bigoplus_{1 \leq i < j \leq \ell} S_{f_i f_j}\to \cdots \to S_{f_1 \cdots f_\ell} \to 0,
$$
where $\Cech^i(\underline{f};S)=\bigoplus \limits_{1 \leq j_1<\ldots<j_i\leq \ell} S_{f_{j_1}\cdots f_{j_i}}$, and on each summand, the homomorphism 
$\Cech^i(\underline{f};S)\to \Cech^{i+1}(\underline{f};S)$
is a localization map with an appropriate sign. For instance, if $\ell=2$, the complex is

$$
0\to S\to S_{f_1}\oplus S_{f_2}\to S_{f_1 f_2} \to 0,
$$
where
$S\to S_{f_1}\oplus S_{f_2}$ maps $s\mapsto (\frac{s}{1},\frac{s}{1})$ and 
$S_{f_1}\oplus S_{f_2}\to S_{f_1 f_2}$ maps $\left(\frac{s}{f^\alpha_1},\frac{r}{f^\beta_2}\right)\mapsto \frac{s}{f^\alpha_1}-\frac{r}{f_2^\beta}.$
We point out that 
$$\Cech^\bullet(\underline{f};S)=\lim\limits_{\overset{\longrightarrow}{t}} \cK(\underline{f}^t;S),$$
 where $\cK(\underline{f}^t;S)$ is the \emph{Koszul complex} associated to $f^t_1\ldots,f^t_\ell$,
 and the map $\cK(\underline{f}^t;S)\to \cK(\underline{f}^{t+1};S)$
 is induced by $S/(f^t_1,\ldots f^t_\ell)\overset{\cdot f_1\cdots f_\ell}{\longrightarrow} S/(f^{t+1}_1,\ldots f^{t+1}_\ell)$.
 
\begin{definition}[Local cohomology]
\label{defn:cech complex defn}
Let $I = (f_1, \ldots, f_\ell)$ be an ideal of a Noetherian ring $S$, let $M$ be an $S$-module, and fix $i \in \NN$. 
We define the \emph{$i$-th local cohomology of $M$ with support in $I$}, denoted $H^i_I(M)$, as the $i$-th cohomology of the complex $\Cech^\bullet(\underline{f};S)\otimes_S M$; \emph{i.e.},
$$
H^i_I(M):=H^i(\Cech^\bullet(\underline{f};S)\otimes_S M)=
\frac{
\Ker
\left( 
\Cech^i(\underline{f};S) \otimes_S M \to \Cech^{i+1}(\underline{f};S) \otimes_S M
\right)
}
{
\IM\left( \Cech^{i-1}(\underline{f};S) \otimes_S M \to \Cech^{i}(\underline{f};S) \otimes_S M \right)
}
$$
\end{definition}

\begin{remark}
The local cohomology module $H^i_I(M)$ does not depend on the choice of generators for $I$, $f_1,\ldots,f_\ell$.
Moreover, it depends only on the radical of $I$; \emph{i.e.}, $H^i_I(M)=H^i_{\sqrt{I}}(M)$.
For this reason, if $X=\Spec{S}$ is the affine scheme associated to $S$ and $Z=\mathbb{V}(I)=\{P\in X\mid I\subseteq P\}$ is the Zariski closed
subset defined by $I,$ then $H^i_I(M)$ can be written as $H^i_Z(M).$ 
\end{remark}

Note that there are several ways to define local cohomology. 
In fact, the definition we chose is not the most natural, but will be advantageous in our discussion due to the interactions between the cited complex $\Cech^\bullet(\underline{f};S)$ and $D$-modules (see Section \ref{Sec D-modules}). 
The local cohomology module $H^i_I(M)$ can also be defined as the $i$-th right derived functor of $\Gamma_I(-)$, where $\Gamma_I(M)=\{v\in M \mid I^j v=0 \hbox{ for some }j\in\NN\}$.  It can also be defined as 
the direct limit $\lim\limits_{\overset{\longrightarrow}{t}} \Ext^i_S(S/I^t,M)$, whose maps are induced by the natural surjections $S/I^{t+1} \twoheadrightarrow S/I^t$. 
%

\edit{Next, we collect some basic facts on local cohomology modules that will be used later.}
\begin{remark}[Graded Local Duality]  \label{LocalDuality}
\edit{Let $K$ be a field, and consider $R=K[x_0,\dots,x_n]$ under the standard grading; \emph{i.e.}, $\deg(x_i)=1$ for $0 \leq i \leq n$. 
If $f$ is a homogeneous polynomial in $R$ and $M$ is a graded $R$-module, then $M_f$ is naturally a graded $R$-module. 
Hence, it follows from Definition \ref{defn:cech complex defn} that each local cohomology module $H^j_I(M)$ is naturally graded whenever $M$ is graded and $I$ is homogeneous.
If $\fm = (x_0, \ldots, x_n)$ is the homogeneous maximal ideal of $R$, then for each finitely generated graded $R$-module $M$, there is a functorial isomorphism}
$$H^t_{\fm}(M)_\ell\cong \Hom_K\left(\Ext^{n+1-t}_R(M,R(-n-1))_{-\ell},K\right)$$
for all integers $t$ and $\ell$ (\emph{cf.} \cite[13.4.6]{BrodmannSharpLocalCohomology}).

For any graded $R$-module $M$, the graded Matlis dual $\D(M)$ of $M$ is defined to be the graded $R$-module defined by $\D(M)_\ell=\Hom_K(M_{-\ell},K)$.   Graded Local Duality states that there is a degree-preserving isomorphism
\[H^t_{\fm}(M)\cong \D\left(\Ext^{n+1-t}_R(M,R(-n-1))\right)\]
for all integers $t$ and all finitely generated graded $R$-modules $M$.  See \cite[13.3 and 13.4]{BrodmannSharpLocalCohomology} for details.
\end{remark}

\begin{remark}[Connection between local cohomology and sheaf cohomology] \label{SheafCohomology}
\edit{Take $R$ as in Remark \ref{LocalDuality}.} Let $M$ be a finitely generated graded $R$-module, and let $\widetilde{M}$ be the sheaf on $\mathbb{P}^n$ associated to $M$. Then there are a functorial isomorphisms (\emph{cf.} \cite[A4.1]{EisenbudCommutativeAlgebra}, 
\cite[Lecture 13]{24HoursLocalCohomology}) 
$$H^t_{\fm}(M)\cong \bigoplus_{\ell\in \mathbb{Z}}H^{t-1}(\mathbb{P}^n,\widetilde{M}(\ell))\ {\rm when}\ t\geq 2,$$
and an exact sequence (functorial in $M$) of degree-preserving maps
$$0\to H^0_{\fm}(M)\to M\to \bigoplus_{\ell\in \mathbb{Z}}H^0(\mathbb{P}^n,\widetilde{M}(\ell))\to H^1_{\fm}(M)\to 0.$$
\end{remark}

\begin{remark}[Mayer-Vietoris sequence for local cohomology]  \label{MVSequence}
\edit{Given} ideals $I$ and $J$ of a Noetherian ring $R$, and an $R$-module $M$, there exists a long exact sequence
\[ \xymatrix@R=.70cm@C=.35cm{ \ar[r] 
0 \ar[r] & H^{0}_{I+J}(M) \ar[r] & H^{0}_I(M) \oplus H^{0}_J(M) \ar[r] &H^0_{I \cap J}(M) \connarrow & \\ 
& H^1_{I+J}(M) \ar[r] & H^1_I(M) \oplus H^{1}_J(M) \ar[r] & H^1_{I \cap J}(M) \ar[r] & \cdots,
} \]
functorial in $M$.
\end{remark}

\begin{remark}
The local cohomology modules $H^i_I(M)$ are usually not finitely generated, even when $M$ is. For instance, if $(S,\fm,K)$ is an $n$-dimensional regular local ring, then
$H^n_\fm(S)\cong E_S(K)$, the injective hull of $K$ over $S$, which is not finitely generated unless $S$ is a field.  
\end{remark}
\subsection{$D$-modules}\label{Sec D-modules}
Let $R$ be a Noetherian ring, and
let $S=R \llbracket x_1,\ldots,x_n \rrbracket $.  
Let $D(S,R)$ denote the ring of $R$-linear differential operators of $S$,  the subring of $\End_R(S)$ given by
$$
D(S,R)=S\left<\frac{1}{t!}\frac{\partial^t}{\partial x^t_1},\ldots,\frac{1}{t!}\frac{\partial^t}{\partial x^t_n}\right>_{t\in\NN},
$$
where $\frac{1}{t!} \frac{\partial^t}{\partial x^t_i}$ is the $R \llbracket x_1,\ldots,x_{i-1},x_{i+1},\ldots,x_n \rrbracket $-linear endomorphism of $S$ induced by
$$
\frac{1}{t!}\frac{\partial^t}{\partial x^t_i} \cdot x^v_j= \begin{cases}
0 & \textrm{if } t>v, \textrm{and} \\
\binom{v}{t} \ x^{v-t}& \textrm{otherwise}.
\end{cases}
$$
For example, when $t=1$, $\frac{\partial^t}{\partial x^t_i} f = \frac{\partial}{\partial x_i} f$ is the formal derivative of $f$ with respect to $x_i.$

Since $D(S,R)\subseteq \End_R(S)$, $\theta \in D(S,R)$ acts on $s \in S$ by $\theta\cdot s=\theta(s).$ In addition, $D(S,R)$ also acts naturally on $S_f$
for every $f\in S$, and with this action, the localization map $S\to S_f$ is a homomorphism of $D(S,R)$-modules 
(see \cite{LyubeznikFinitenessLocalCohomologyModules,LyubeznikFreeChar} for details). 
If $t=1,$ this action is given by the quotient rule for derivatives.
As a consequence,
the local cohomology modules are $D(S,R)$-modules. 
Moreover, for all $i_1,\ldots,i_\ell \in \NN$ and all ideals $I_1,\ldots,I_\ell$ of $S$, the iterated local cohomology modules $H^{i_\ell}_{I_\ell}\cdots H^{i_2}_{I_2}  H^{i_1}_{I_1} (S)$
are also $D(S,R)$-modules  \cite{LyubeznikFinitenessLocalCohomologyModules}.

\edit{When $R=K$ is a field of characteristic zero, it was proved in \cite{BjorkRingsDifferentialOperators} that each localization $S_f$ has finite length as a $D(S,K)$-module.
Consequently, each local cohomology module of the form $H^{i_\ell}_{I_\ell}\cdots H^{i_2}_{I_2} H^{i_1}_{I_1} (S)$, where 
$i_1,\ldots,i_\ell \in \NN$ and $I_1,\ldots,I_\ell$ are ideals of $S$,
 also has finite length as a $D(S,K)$-module. 
In \cite{LyubeznikFinitenessLocalCohomologyModules}, Lyubeznik used this fact to show that the Bass numbers $\dim_K \Ext^i_S(K,H^j_I(S))$ are finite for all $i,j\in \NN$ and all ideals $I\subseteq S$. Subsequently, the generalized these results to fields of positive characteristic in \cite{LyubeznikFModulesApplicationsToLocalCohomology,LyubeznikFreeChar}.}

\subsection{A key functor}\label{Key Functor}
Here, we introduce a functor motivated by work of Kashiwara and 
used by Lyubeznik to prove that the Lyubeznik numbers are well defined
\cite{LyubeznikFinitenessLocalCohomologyModules}. 
Further properties of this functor were developed by the first two authors to define the generalized  Lyubeznik numbers \cite{NuWiEqual}.
Let $K$ be a field, let $R=K \llbracket x_1,\ldots, x_n \rrbracket$, and let $S=R \llbracket x_{n+1} \rrbracket$.
Let $G$ denote the functor from the category of $R$-modules to that of $S$-modules given by, for an $R$-module $M$, 
\[G(M)= M \otimes_R S_{x_{n+1}}/S. \] 
This functor satisfies the following properties; we refer to \cite[Section 3]{NuWiEqual} for details.
{\edit{
\begin{properties} \label{FunctorProperties} With $G$ the functor defined above, the following hold.
\begin{enumerate}
\item[(1)] $G$ is an equivalence of categories from the category of $D(R,K)$-modules to that of $D(S,K)$-modules supported 
at $\mathbb{V}(x_{n+1}S)$, the Zariski closed set given by $x_{n+1}$, with inverse functor $N \mapsto \Ann_N (xS)$;
\item[(2)] $M$ is a finitely generated $R$-module if and only if $G(M)$ is a finitely generated $D(S,\edit{K})$-module;
\item[(3)] $M$ is an injective $R$-module if and only if $G(M)$ is an injective $S$-module;
\item[(4)] $\Ass_S G(M)=\{(P,x_{n+1})S\mid P\in \Ass_R M\}$;
\item[(5)] $G\left(H^j_I(M)\right)=H^{j+1}_{(I,x)S}(M\otimes_R S)$
for every ideal $I\subseteq R$ and every $j\in \NN$;
\item[(6)] $G\left(\itLC{j}{s}{M}\right) \cong
H^{j_s}_{(I_s,x_{n+1})S} \cdots H^{j_2}_{(I_2,x_{n+1})S}H^{j_1+1}_{(I_1,x_{n+1})S}(M\otimes_R S)$
for all ideals $I_i$ of $R$ and all $j_i \in \NN$, $1\leq i \leq s$; and
\item[(7)] $\Ext^{i}_S (M,G(N))=\Ext^{i}_R (M,N)$
for all $R$-modules $M$ and $N$, and all $i, j \in \NN$.
\end{enumerate}
Moreover, this functor preserves certain properties over rings of differential operators. Namely,
\begin{enumerate}
\item[($1'$)] $M$ is a finitely generated $D(R,K)$-module if and only if $G(M)$ is a finitely generated $D(S,K)$-module, and
\item[($2'$)] $\Length_{D(R,K)}M=\Length_{D(S,K)} G(M).$
\end{enumerate}
\end{properties}
}}

\section{Definition and first properties}\label{Definition}
Referring to local cohomology modules, which were reviewed in Section \ref{PreLocalCohomology},
we may now define the Lyubeznik numbers
\cite{LyubeznikFinitenessLocalCohomologyModules}.

\begin{ThmDef}[Lyubeznik numbers] \label{LyuDef}
Let $(R, \fm, K)$ be a local ring admitting a surjection from an $n$-dimensional regular local ring $(S,\eta,K)$ 
containing a field, $\pi:S\surj R$.  Let $I = \ker (\pi).$
The \emph{Lyubeznik number of $R$ with respect to $i,j \in \NN$} is defined as
$$
\lambda_{i,j}(R):=\Dim_K \Ext^i_S\left(K,H^{n-j}_I (S)\right),
$$
and depends only on $R$, $i$, and $j$; \emph{i.e.}, this number is independent of the choice of $S$ and of $\pi$.  
If $(R, \fm, K)$ is \emph{any} local ring containing a field, then $\lambda_{i,j} (R) := \lambda_{i,j} (\widehat{R}),$ where $\widehat{R}$ is the completion of $R$ with respect to $\fm$. 
\end{ThmDef}
\begin{proof}[Sketch of proof]
We have that $\dim_K\Ext^{i}_S(K,H^{n-j}_{I} (S))$ is finite \cite{LyubeznikFinitenessLocalCohomologyModules,Huneke-Sharp}, 
so it remains to prove that these numbers are well defined.
Let $\pi' : S'\surj R$ be another surjection, where $S'$ is a regular local ring of dimension $n'$ that contains a field.
Set $I' = \Ker(\pi')$, and let $\fm'$ be the maximal ideal of $S'$.
Since the Bass numbers with respect to the maximal ideal are not affected by completion, we may assume that
$R$, $S$, and $S'$ are complete, so that we can take $S =  K \llbracket x_1,\ldots,x_{n} \rrbracket $
 and $S'= K \llbracket y_1,\ldots,y_{n'} \rrbracket $. 
Let $S''=K \llbracket z_1,\ldots,z_{n+n'} \rrbracket $, and 
let $\pi'': S''\surj R$ be the surjective map defined by $\pi''(z_j)=\pi(x_j)$ for $1\leq j\leq n$
and $\pi''(z_j)=\pi'(y_{j-n})$ for $n\leq j\leq n+n'$.
Let $I''$ be the preimage of $I$ under $\pi''$, respectively. 
By using Properties  \ref{FunctorProperties} {\edit{($5$) and ($7$),}}
we obtain:
\begin{align*}
\dim_K\Ext^{i}_{S''} (K,H^{n+n'-j}_{I''}  (S''))&=\dim_K\Ext^{i}_S (K,H^{n-j}_{I} (S)) \text{ and} \\
\dim_K\Ext^{i}_{S''} (K,H^{n+n'-j}_{I''}  (S''))&=\dim_K\Ext^{i}_{S'} (K,H^{n'-j}_{I'} (S')),
\end{align*}
and we are done.
\end{proof}

Some basic properties of the Lyubeznik numbers are the following \cite{LyubeznikFinitenessLocalCohomologyModules}.

\begin{properties} \label{FirstProp}
If $(R, \fm, K)$ is a $d$-dimensional local ring containing a field, then the following hold.  
\begin{enumerate}
\item $\lambda_{i,j} (R)=0 $ if either $j>d$ or $i>j$,
\item $\lambda_{d,d}(R)\neq 0$, and
\item $\lambda_{d,d}(R)=1$ if $R$ is analytically normal.
\end{enumerate}
\end{properties}
\noindent Due to Property \ref{FirstProp} (1), we can record all nonzero Lyubeznik numbers in the following matrix.
\begin{definition}[Lyubeznik table] \label{LyubeznikTable}
The \emph{Lyubeznik table of $R$} is defined as the following $(d+1) \times (d+1)$ matrix:
\[ 
{\Lambda}(R) := \left( \lambda_{i,j}(R) \right)_{0 \leq i, j \leq d} = \begin{pmatrix}
\lambda_{0,0}(R) &\lambda_{0,1}(R) & \cdots  & \lambda_{0,d-2}(R) & \lambda_{0,d-1}(R) & \lambda_{0,d}(R) \\
0 & \lambda_{1,1}(R) &  \cdots & \lambda_{1,d-2}(R)  & \lambda_{1,d-1}(R)  & \lambda_{1,d}(R) \\
0 & 0 & \ddots  & \vdots  &  \vdots & \vdots \\
\vdots & \vdots &  \ddots & \lambda_{d-2,d-2}(R) & \lambda_{d-2,d-1}(R)  & \lambda_{d-2,d}(R) \\
0 & 0 &  \cdots & 0 & \lambda_{d-1,d-1}(R)  & \lambda_{d-1,d}(R) \\
0 & 0 &  \cdots & 0 & 0 & \lambda_{d,d}(R) \\
\end{pmatrix}.
\] \end{definition}
\noindent Note that all entries below the diagonal are zero by Properties \ref{FirstProp} (1).  

{
There are algorithms to compute Lyubeznik numbers in certain cases 
in characteristic zero, which rely on the $D$-module structure of local cohomology modules
\cite{WaltherAlgorithm,AlvarezMontanerLeykin}. }
\section{The highest Lyubeznik number}
Some basic properties of the Lyubeznik numbers are that, for a $d$-dimensional local ring $R$ that contains a field,  $\lambda_{d,d}(R)\neq 0$, and
$\lambda_{i,j}(R)=0$ if either $i> d$ or $j> d$. 
(see Properties \ref{FirstProp}).

Hence, the following definition is natural.
\begin{definition}[Highest Lyubeznik number] \label{Highest}
Given a $d$-dimensional local ring $R$ that contains a field, $\lambda_{d,d}(R) = 1$ is called the \emph{highest Lyubeznik number of $R$}.
\end{definition}

Consider the following simple example.
\begin{example} \label{CI}
\edit{Let $R$ be a $d$-dimensional complete local ring containing a field. Suppose that $R$ can be written as $S/I$ where $(S,\mathfrak{m})$ is a regular local ring that contains field and $I$ is an ideal of $S$ such that 
\begin{equation}
\label{defn:CCI}
H^j_I(S)\neq 0\ \text{if and only if}\ j=\height(I).
\end{equation} 
One can see that the spectral sequence $E^{p,q}_2 =H^p_{\mathfrak{m}}(H^q_I(S)) \underset{p}{\Rightarrow} H^{p+q}_{\mathfrak{m}}(S) = E^{p,q}_\infty$ degenerates at the $E_2$-page and, hence, we have that
\begin{equation}\label{Trivial LT}
\lambda_{i,j}(R)=\begin{cases}1& \text{if } i=j=d, \text{and} \\ 0& {\rm otherwise}. \end{cases}
\end{equation}
}
\end{example}

Ideals $I$ that satisfy (\ref{defn:CCI}) are called {\it cohomologically complete intersection ideals} \cite{HellusSchenzel}. 
Examples of cohomologically complete intersection ideals include ideals generated by regular sequences, and Cohen-Macaulay ideals in positive characteristic. 

We say that a ring has \emph{a trivial Lyubeznik table} if the entries of its Lyubeznik table follow the formula given in \eqref{Trivial LT}.  There are rings that are not cohomologically complete intersection and still have a trivial Lyubeznik table, {\it e.g.} sequentially Cohen-Macaulay rings in prime characteristic \cite{AM-SCM}.

Example \ref{CI} indicates that the highest Lyubeznik number, $\lambda_{d,d}(R)$, stands out in that it is never zero, but all other Lyubeznik numbers may vanish. {\edit{Our goal of this section is to give a 
characterization of $\lambda_{d,d}(R)$.}} To this end, we will consider the Hochster-Huneke graph of a local ring \cite[Definition 3.4]{HochsterHunekeIndecomposable}.

\begin{definition}[Hochster-Huneke graph]
Let $B$ be a local ring. The \emph{Hochster-Huneke graph $\Gamma_B$ of $B$} is defined as follows:  its vertices are the top-dimensional minimal prime ideals of $B$, and two distinct
 vertices $P$ and $Q$ are joined by an edge if and only if $\height_B(P+Q)=1$. 
\end{definition}  

Before we can state the main result in this section, we want to point out that, since $\lambda_{i,j}(R)$ does not change under any faithfully flat extension, one {\it cannot} use topological information about $\Spec(R)$ to characterize $\lambda_{d,d}(R)$. It turns out  that one must consider a faithfully flat extension $B$ of $R$ that is complete and has a separably closed residue field ({\it{a fortiori}} strictly Henselian); more specifically, one may choose $B=\widehat{R^{sh}}$, the completion of the strict Henselization of $R$.   
(As pointed out in \cite{LyubeznikSomeLocalCohomologyInvariants}, the graph $\Gamma_B$ with $B=\widehat{R^{sh}}$ can be realized by a substantially smaller ring.
Indeed, if $k$ is a coefficient field of $\widehat{R}$, then there exists a finite separable extension $K$ of $k$ such that the graphs $\Gamma_B$ and $\Gamma_{\widehat{R}\otimes_kK}$ are isomorphic. Since this finite separable extension $K$ of $k$ is not explicitly constructed, we choose to work with $B=\widehat{R^{sh}}$.)

To get a sense of the difference between the graph of $R$ and that of $B=\widehat{R^{sh}}$, we consider the following example.
\begin{example}
Let $S=\mathbb{Q}[X,Y,Z,W]_{(X,Y,Z,W)}$, define \[I=(Z^2-3X^2,W^2-3Y^2,ZW-3XY,XW-3YZ),\] and let $R=S/I$. 
Let $x,y,z$, and $w$ denote the images of $X,Y,Z,$ and $W$ in $R$, respectively. One can verify that 
$R\cong \mathbb{Q}\left[x,y,\sqrt{3}x,\sqrt{3}y\right]_{(x,y,\sqrt{3}x,\sqrt{3}y)}$; hence, $R$ is a domain. 
Therefore, the Hochster-Huneke graph of $R$ consists solely of one vertex, so that $\Gamma_R$ has only one connected component.\par
On the other hand, $B=\widehat{R^{sh}}=\frac{\mathbb{Q}^{sep} \llbracket X,Y,Z,W \rrbracket }{(Z-\sqrt{3}X,W-\sqrt{3}Y)(Z+\sqrt{3}X,W+\sqrt{3}Y)}$. Clearly, there are 2 vertices in the graph of 
 $B$ which correspond to the minimal prime ideals $(Z-\sqrt{3}X,W-\sqrt{3}Y)$ and $(Z+\sqrt{3}X,W+\sqrt{3}Y)$, respectively, and the height of the sum of these ideals is two in $B$. Therefore, 
$\Gamma_B$ consists of two vertices and no edges, and $\Gamma_B$ has two connected components.

It is not hard to check that $\lambda_{2,2}(R)=\lambda_{2,2}(B)=2$.
\end{example}

Now we can state the main result of this section. The following theorem was first proved by Lyubeznik in characteristic $p>0$ in \cite{LyubeznikSomeLocalCohomologyInvariants}, and later a characteristic-free proof using different methods was given in \cite{ZhanghighestLyubeznikNumber}.
\begin{theorem}
\label{main theorem in section of highestNumber}
Let $R$ be a $d$-dimensional local ring that contains a field. Then $\lambda_{d,d}(R)$ is equal to the number of connected components of the Hochster-Huneke graph of $\widehat{R^{sh}}$.
\end{theorem}

\begin{proof}[Sketch of proof] Let $B=\widehat{R^{sh}}$.  Since $\lambda_{d,d}(R)$ does not change under faithfully flat extensions, we may replace $R$ by $B$, and assume that $R$ is a complete local ring that contains a separably closed coefficient field. 
Let $\Gamma$ denote the Hochster-Huneke graph of $R$. 

Let $\Gamma_1,\dots,\Gamma_t$ denote the connected components of $\Gamma$. For $1 \leq j \leq t$, let $I_j$ be the intersection of the minimal primes of $R$ that are vertices of $\Gamma_j$.  Using the Mayer-Vietoris sequence of local cohomology (see Remark \ref{MVSequence}), one can prove that 
\[\lambda_{d,d}(R)=\sum_{j=1}^t\lambda_{d,d}(R/I_j).\]
Hence, we are reduced to proving that $\lambda_{d,d}(R)=1$ when $R$ is equidimensional (since each $R/I_j$ is) and $\Gamma$ is connected.

We now proceed by induction on $d:=\dim(R)$. When $d =2$, the theorem has already been established by Walther in \cite{WaltherLyubeznikNumbers} using the ``Second Vanishing Theorem" of local cohomology \cite[Theorem 1.1]{HunekeLyuVanishing}.

Suppose now that $d\geq 3$, and assume that the theorem holds for all rings of dimension less than $d$. 
Since $R$ is complete, we can write $R=S/I$, where $(S,\mathfrak{m})$ is an $n$-dimensional complete regular local ring that contains a separably closed field, and $I$ is an ideal of $S$. We will choose a special element $s \in S$ as follows:  
Note that there are only finitely many minimal elements of $\Supp_S(H^{n-d+1}_I(S))$  (\cite[Corollary 3.6]{LyubeznikFinitenessLocalCohomologyModules}). 
If $\Ass_S \(H^{n-d+1}_I(S)\)\neq \{\mathfrak{m}\}$, then by prime avoidance, we can fix an element $s \in \mathfrak{m}$ that is not in any minimal prime of $I$, nor in any minimal element of $\Supp_S\(H^{n-d+1}_I(S)\)$. On the other hand, if $\Ass_S\(H^{n-d+1}_I(S)\) = \{\mathfrak{m}\}$, then fix $s \in \mathfrak{m}$ that is not in any minimal prime of $I$.

Note that by our choice of $s$, if  $\bar{s}$ denotes the image of $s$ in $R$, 
 the ring $R/\bar{s}R$ is equidimensional and $\dim(R/\bar{s}R)=d-1$.  
 Moreover, one can prove that $\bar{s}$ satisfies the following two properties:
\begin{enumerate}
\item $\lambda_{d,d}(R)=\lambda_{d-1,d-1} \left(R / \sqrt{\bar{s}R}\right)$, and
\item The Hochster-Huneke graph of $R/\sqrt{\bar{s}R}$ is connected.
\end{enumerate}
Once these are established, we have $\lambda_{d,d}(R)=\lambda_{d-1,d-1}\left(R/\sqrt{\bar{s}R}\right)=1$ by the inductive hypotheses, completing the sketch.
\end{proof}

As an application of Theorem \ref{main theorem in section of highestNumber}, we will give a much simpler proof of the following theorem, originally proved by Kawasaki in \cite{KawasakiHighestLyubeznikNumber} using a spectral sequence argument. 
\begin{theorem}
\label{S2 implies lambda to be 1}
Let $(S,\mathfrak{m})$ be a regular local ring that contains a field, let $I$ be an ideal of $S$, and let $d = \dim(S/I)$.  If $S/I$ satisfies Serre's $S_2$-condition, then $\lambda_{d,d}(S/I)=1$.
\end{theorem}
\begin{proof}
Let $R = S/I$, and let $B=\widehat{R^{sh}}$.
According to Theorem \ref{main theorem in section of highestNumber}, it suffices to show that the Hochster-Huneke graph $\Gamma_B$ is connected. As $S$ is regular, so is also Cohen-Macaulay, it is universally catenary ({\it cf.} \cite[Theorem 17.9]{MatsumuraCommutativeRingTheory}).  Thus, $B$ is catenary. Since $R$ satisfies the $S_2$ condition, so does $B$. Then \cite[Remark 2.4.1]{HartshorneCompleteIntersectionConnectedness} implies that $B$ is equidimensional. Therefore, $\Gamma_B$ must be connected by \cite[Theorem 3.6]{HochsterHunekeIndecomposable}.
\end{proof}

It turns out that the converse of Theorem \ref{S2 implies lambda to be 1} does not hold, as the following example from \cite{KawasakiHighestLyubeznikNumber} illustrates.
\begin{example}\label{ExMon1}
Let $K$ be a field, and let $S=K[x_1,x_2,x_3,x_4,x_5,x_6]_{(x_1,\dots,x_6)}$; moreover, let 
\begin{equation} \label{exmon1} I=(x_1,x_2,x_3)\cap (x_2,x_3,x_4)\cap (x_3,x_4,x_5)\cap (x_4,x_5,x_6)\cap (x_5,x_6,x_1).
\end{equation}
It is clear that $\dim(S/I)=3$, and one can check that $\lambda_{3,3}(S/I)=1$. 
If we set $P=(x_1,x_2,x_3,x_5,x_6)$, then the depth of $(S/I)_P$ is 1, but the height of $P$ in $S/I$ is 2. Hence, $S/I$ does not satisfy the $S_2$ condition.
\end{example}

{
We end this section by pointing out a related recent result of Schenzel \cite{Schenzel2011}. Suppose that $S$ is a complete regular local ring of equal characteristic and that $I$ is an ideal of $S$ of height $c$, and let $d=\dim(S/I)$.
Schenzel proved that $B:=\Hom_S(H^c_I(S),H^c_I(S))$ is a free $S$-module of rank $\lambda_{d,d}(S/I)$ \cite{Schenzel2011}. }
\section{Lyubeznik numbers for projective schemes} \label{ProjSchemesSec}

Throughout this section, we fix the following notation.

\begin{notation}
Fix a projective scheme $X$ over a field $K$.
Let  $S=K[x_0,\cdots,x_n]$, and consider the standard grading on $S$; \emph{i.e.}, $\deg(x_i)=1$ for $0 \leq i \leq n$.
Let $\fm$ denote is homogeneous maximal ideal, and let $I$ be a homogeneous ideal of $S$.
Under an embedding $\iota: X\hookrightarrow \mathbb{P}^n_K$, we may write $X=\Proj(S/I)$.
Let $R =(S/I)_{\fm}$, the local ring at the vertex of the affine cone over $X$, which clearly contains a field.
\end{notation}

In our setting, we may consider the Lyubeznik numbers of $R$. Theorem \ref{main theorem in section of highestNumber} has the following interesting consequence.

\begin{theorem} \label{Consequence}
In our setting, if $K^{sep}$ denotes the separable closure of $K$, let $X_1,\dots,X_t$ denote the $d$-dimensional irreducible components of $X^{sep}:=X\times_K K^{sep}$.  Let $\Gamma_X$ be the graph on vertices $X_1,\dots,X_t$, where $X_i$ and $X_j$ are joined by an edge if and only if $\dim(X_i\cap X_j)=d-1$. 
Then $\lambda_{d+1,d+1}(R)$ is the number of connected components of $\Gamma_X$. Consequently, $\lambda_{d+1,d+1}(R)$ is an invariant of $X$, independent of the choice of embedding $\iota$.   
\end{theorem}

%

This theorem prompts the following question.

\begin{question}
\label{independence of L numbers of projective schemes}
In our setting, is it true that $\lambda_{i,j}(R)$ depends only on $X,i$, and $j$; \emph{i.e.}, it is independent of the choice of embedding $\iota: X\hookrightarrow \mathbb{P}^n_K$? 
\end{question} 

\noindent When $K$ has characteristic $p>0$, the answer to Question \ref{independence of L numbers of projective schemes} turns out to be affirmative\footnote{Recently, it has been proved in \cite{Switala} that Question \ref{independence of L numbers of projective schemes} has a positive answer for a nonsingular projective variety $X$ in characteristic zero.}. The main goal of this section is to sketch the proof of the following theorem \cite[Theorem 1.1]{ZhangLyubeznikNumbersProjSchemes}.

\begin{theorem}
\label{independence in char p}
In our setting, assume that $K$ has characteristic $p>0$. Then each $\lambda_{i,j}(R)$ depends only on $X,i$, and $j$; in particular, it does not depend on the choice of embedding $\iota: X\hookrightarrow \mathbb{P}^n_K$.
\end{theorem}


%

Noting Remark \ref{SheafCohomology}, one can see that if the index $i$ equals zero or one, the connection between local cohomology and sheaf cohomology is more complicated than in the other cases; consequently, the proof of Theorem \ref{independence in char p} is more technical. 

For this reason, we will sketch the proof of  Theorem \ref{independence in char p} in the case that $i \geq 2$, and will, accordingly, make this assumption for the rest of the section.  
Since field extensions do not affect the Lyubeznik numbers, we will also assume that $K$ is algebraically closed. For each $S$-module $M$, set 
\[\cE^{i,j}(M):=\Ext^{n+1-i}_S\left(\Ext^{n+1-j}_S(M,S),S\right).\]
{\edit{Let $\dualizing$ denote the dualizing complex on $X$ \cite{HartshorneResidues} }}induced by the structure morphism $X\to K$, and 
for each sheaf of $\cO_X$-modules $\mathcal{G}$, set
\begin{equation*} \E^{i,j}(X,\mathcal{G}):=\Ext^{1-i}_X\left(\cE xt^{1-j}(\mathcal{G},\dualizing),\dualizing\right). \end{equation*}
Consider the following characteristic-free result.
\begin{proposition} \label{zeroIndependent}
The degree zero piece of $\cE^{i,j}(S/I)$, $\cE^{i,j}(S/I)_0$, depends only on $X,i$, and $j$; in particular, it is independent of the choice of embedding $\iota: X\hookrightarrow \mathbb{P}^n_k$.
\end{proposition}
\begin{proof}
If $\dualizing$ denotes the dualizing complex on $X$ induced by the structure morphism $X\to K$, 
we have the following isomorphisms.
\begin{alignat*}{3}
\cE^{i,j}(S/I)_0 &= \Ext^{n+1-i}_S\left(\Ext^{n+1-j}_S(S/I,S(-n-1)),S(-n-1)\right)_0 & \ \ \ \  \ \ \ \ &  \\
&\cong \D\left(H^i_{\fm}\(\Ext^{n+1-j}_S\(S/I,S(-n-1)\)\)_0\right) && {\rm by \ Remark \ \ref{LocalDuality}} & & \\
&\cong \D\left(H^{i-1}(\mathbb{P}^n_k,\cE xt^{n+1-j}(\cO_X,\omega_{\mathbb{P}^n_k})\right) & &{\rm by \ Remark \ \ref{SheafCohomology} }\\
&\cong \Ext^{n+1-i}_{\mathbb{P}^n_k}\left(\cE xt^{n+1-j}(\cO_X,\omega_{\mathbb{P}^n_k}),\omega_{\mathbb{P}^n_k} \right)\ & & {\rm by\ Serre\ Duality}\\
&\cong \E^{i,j}(X,\cO_X),
\end{alignat*}
where the last isomorphism is a consequence of Grothendieck Duality for finite morphisms.
\end{proof}

Now we make the transition to characteristic $p>0$ setting. 
The fundamental tool in the study of rings $S$ of characteristic $p>0$ is the \emph{Frobenius endomorphism} 
\begin{equation}
F: S  \to S, \ \text{given by} \ F(s) = s^p \label{Frob}
\end{equation}
 and its iterates $F^e:S\to S$ for each $e \geq 1$, given by $F^e(s) = s^{p^e}$. 
 Note that some discussion on singularities defined using these maps appear in Section \ref{generalized}.
  
Very often one must distinguish the source ring from the target ring in the Frobenius endomorphism.  
We will do so by denoting the target ring $S$ by $S^{(e)}$; \emph{i.e.}, $S^{(e)}$ is the same as $S$ as an abelian group, but its $S$-module structure is given via $F^e$:
 for each $s \in S$ and $s'\in S^{(e)}$,  $s \cdot s'=s^{p^e}s'$ . Using this notation, we can introduce the \emph{$e$-th iterated Frobenius functor} (also called the \emph{Peskine-Szpiro functor}), a functor from the category of $S$-modules to itself.  It is denoted $\cF^e$, and given an $S$-module $M$, $\cF^e(M)=S^{(e)}\otimes_SM$.
Since $S$ is a polynomial ring in our setting, it is straightforward to check that $\cF^e$ is an exact functor, as the image of $F^e$ is $S^{p^e}$, and clearly $S$ is a free module over $S^{p^e}$. One can also check the following facts about $\cF^e$.

\begin{proposition}[Basic facts on $\cF^e$]
\label{basic facts on Frobenius}
The Frobenius functor $\cF^e$ satisfies the following properties for all integers $t$ and all $e \geq1$, where each isomorphism is degree-preserving.
\begin{enumerate}
\item $\cF^e(S)\cong S$.
\item $\cF^e(S/I)\cong S/I^{[p^e]}$.
\item $\cF^e\(\Ext^t_S(M,N)\)\cong \Ext^t_S\(\cF^e(M),\cF^e(N)\)$.
\end{enumerate}
\end{proposition}

It follows from Proposition \ref{basic facts on Frobenius} that we have the composition of natural maps
\[\cE^{i,j}(S/I)\to \cF^e(\cE^{i,j}(S/I))=S^{(e)}\otimes_S\cE^{i,j}(S/I)\xrightarrow{\sim} \cE^{i,j}(S/I^{[p^e]})\to \cE^{i,j}(S/I),\]
where the first map is given by $z\mapsto 1\otimes z$ for each $z\in \cE^{i,j}(S/I)$, and the last map is induced by the natural surjection $S/I^{[p^e]}\surj S/I$. We can restrict the composition to $\cE^{i,j}(S/I)_0$, and we call the restriction $\alpha$. Note that $\alpha$ is not $K$-linear, but $\alpha$ satisfies $\alpha(cz)=c^p\alpha(z)$ for every $c\in K$ and $z\in \cE^{i,j}(S/I)_0$. Such a map is called a {\it $p$-linear structure}. Given this $p$-linear structure $\alpha$ on $\cE^{i,j}(S/I)_0$, we consider its stable part,
\[\left(\cE^{i,j}(S/I)_0\right)_s:=\bigcap_{e\geq 1}\alpha^e\left(\cE^{i,j}(S/I)_0\right).\]

\begin{proof}[Sketch of the proof of Theorem \ref{independence in char p} when $i\geq 2$]
One can prove that 
\[\lambda_{i,j}(R)=\dim_K \(\cE^{i,j}(S/I)_0\)_s.\]
Since $\cE^{i,j}(S/I)_0$ is independent of the embedding by Proposition \ref{zeroIndependent}, it remains to show that the $p$-linear structure, $\alpha$, on $\cE^{i,j}(S/I)_0$ is also independent of the embedding. To this end, one can consider the the natural $p$-linear structure $\beta$ on $\E^{i,j}(X,\cO_X)$ induced by the Frobenius endomorphism $F_X:X\to X$, and prove that the diagram
\[\xymatrix{
\cE^{i,j}(S/I)_0 \ar[r]^{\sim} \ar[d]^{\alpha}& \E^{i,j}(X,\cO_X)\ar[d]^{\beta} \\
\cE^{i,j}(S/I)_0 \ar[r]^{\sim} & \E^{i,j}(X,\cO_X)
}\] commutes.  
Thus, \[\dim_K\(\cE^{i,j}(S/I)_0\)_s=\dim_K \E^{i,j}(X,\cO_X)_s,\] 
where $\E^{i,j}(X,\cO_X)_s$ is the stable part of $\E^{i,j}(X,\cO_X)$ under $\beta$, which depends only on $X,i$, and $j$, completing the proof.
\end{proof}
\section{Further geometric and topological properties}

Suppose that $V$ is a scheme of finite type over $\CC$ with an isolated singularity at $p\in V$; moreover, let $R = \mathcal{O}_{V,p}$.  
Lyubeznik remarked in \cite{LyubeznikFinitenessLocalCohomologyModules} that by combining results of Ogus on connections of local cohomology modules with de Rham cohomology in \cite{OgusLocalCohomologicalDimension} with results of Hartshorne relating de Rham cohomology and singular cohomology \cite{HartshorneOntheDeRhamCohomology}, one can show that if $j\leq \dim(R) - 1$, then $\lambda_{0, j}(R) = \Dim_{\CC} H^j_p(V,\mathbb{C})$, where $H^j_p(V,\mathbb{C})$ denotes the $j$-th singular cohomology group of $V$ with complex coefficients and support in $p$.   

In \cite{GarciaSabbah}, Garc\'ia L\'opez and Sabbah extend this result, computing all the Lyubeznik numbers in this case in terms of vector space dimensions of singular cohomology groups. This work employs the Riemann-Hilbert correspondence and duality for holonomic $D$-modules.

\begin{theorem}  \label{GS} Let $V$ be a scheme of finite type over $\CC$ with an isolated singularity at the point $x_0$.  Let $R= \mathcal{O}_{V,x_0}$ and let $d = \dim(R)$.
If $d=1$, then $\lambda_{1,1}(R) = 1$ and all other $\lambda_{i,j}(R)=0$.  On the other hand, if $d\geq 2$, then if $H^j_{x_0}(V,\mathbb{C})$ denotes the $j$-th singular cohomology group of $V$ with coefficients in $\mathbb{C}$ and support in $x_0$, the following hold.
\begin{enumerate}
\item $\lambda_{0,j}(R) = \Dim_{\CC} H^j_{x_0}(V,\mathbb{C})$ if $1 \leq j \leq d -1$, 
\item $\lambda_{i,d}(R) = \Dim_{\CC} H^{i+d}_{x_0}(V,\mathbb{C})$ if $2 \leq i \leq d$, and
\item All other $\lambda_{i,j}(R) = 0$.
\end{enumerate}
\end{theorem}

\edit{In fact, part (2) of Theorem \ref{GS} is deduced from part (1) via Poincar\'e duality, as illustrated by its Lyubeznik table:}
\edit{
\[\Lambda(R) = \begin{pmatrix}
0 &  \lambda_{0,1}(R)  & \lambda_{0,2}(R) & \cdots & \lambda_{0, d-1}(R) & 0 \\
0 & 0 & 0 & \cdots & 0 & 0 \\
0 & 0 & 0 & \cdots & 0 & 0 \\
0 & 0 & 0 & \cdots & 0 & \lambda_{0,d-1}(R) \\
0 & 0 &  0 & \cdots & 0 & \lambda_{0,d-2}(R) \\
\vdots & \vdots & \vdots & \ddots & \vdots & \vdots \\
0 & 0 & 0 & \cdots & 0 & \lambda_{0,2}(R) \\
0 & 0 & 0 & \cdots & 0 & \lambda_{0,1}(R) + 1 
\end{pmatrix}. \]
}


\edit{Blickle and Bondu have found interpretations of Lyubeznik numbers in terms of \'etale cohomology \cite{B-B}, which is closely related to Theorem \ref{GS}. Results in \cite{GarciaSabbah} and \cite{B-B} have been generalized in \cite{BlickleLyuCIS}.   
 The result is as follows.

\begin{theorem}
Let $K$ be a separably closed field, and given a closed $d$-dimensional $K$-subvariety $Y$ of a smooth $n$-dimensional variety $X$, let 
let $A=\mathcal{O}_{Y,x}$. 
Suppose that the modules $H^{n-i}_{[Y]}(\mathcal{O}_X)$ are supported in the point $x$ for all $i\neq d$. Then the following hold.
\begin{enumerate}
\item For $2\leq a \leq d$, one has that 
\begin{enumerate}
\item $\lambda_{a,d}(A)  =\lambda_{0,d-a+1}(A)$ for $a\neq d$, 
\item $\lambda_{d,d}(A) = \lambda_{0,1}(A)+1$, and 
\item All other $\lambda_{a,i}(A)=0$.
\end{enumerate}
\item If $\delta_{a,d}$ denotes the Kroneker delta function, then
\[\lambda_{a,d}(R) -\delta_{a,d} = \begin{cases}\dim_{\mathbb{F}_p} H^{d-a+1}_{\{x\}}(Y_{\text{\'et}}, \mathbb{F}_p) & if\ K \ \text{has characteristic} \ p, \ \text{and}\\ \dim_{\mathbb{C}} H^{d-a+1}_{\{x\}}(Y_{\text{an}},\mathbb{C}) & if\ K=\mathbb{C}. \end{cases}\] 
\end{enumerate}
\end{theorem}
}

\section{Generalized Lyubeznik numbers} \label{generalized}

The \emph{generalized Lyubeznik numbers} are a family of invariants associated to a local ring containing a field, which contains the Lyubeznik numbers.  These invariants can capture more subtle information about the ring than the (original) Lyubeznik numbers can. 
Moreover, they can encode information about $F$-singularities in characteristic $p>0$, and have connections to $D$-module characteristic cycle multiplicities in characteristic zero.
Throughout our discussion, we use \cite{NuWiEqual} as our reference. 
The generalized Lyubeznik numbers are defined as follows.  

\begin{ThmDef}[Generalized Lyubeznik numbers] \label{GenDef}
Let $(R, \fm, K)$ be a local ring containing a field  and let $\widehat{R}$ be its completion at $\fm$.   
Given a coefficient field  $L$ of $\widehat{R}$, there exists a surjection $\pi: S \to \widehat{R}$, where $S = K \llbracket x_1, \ldots, x_n \rrbracket $ for some $n \geq 1$, and such that $\pi(K) = L$.
For $1 \leq j \leq \ell$, fix  $i_j \in \NN$ and ideals $I_j \subseteq R$, and let $J_j = \pi^{-1} \left( I_j \widehat{R} \right) \subseteq S$.
The \emph{generalized Lyubeznik number of $R$ with respect to $L$, the $I_j$, and the $i_j$} is defined as  
 $$\lambda^{i_\ell,\ldots, i_1}_{I_\ell,\ldots, I_1}(R;L):=\Length_{D(S,K)} H^{i_\ell}_{J_\ell}  \cdots  H^{i_2}_{J_2}H^{n-i_1}_{J_1}(S).$$
This number is finite and depends only on $R,$ $L$, the $I_j$, and the $i_j$ (\emph{i.e.}, it is independent of the choice of $S$ and of $\pi$).

When $\widehat{R}$ has a unique coefficient field (\emph{e.g}., when $K$ is a perfect field of characteristic $p>0$), or when the choice of $L$ is clear in the context, this invariant is denoted $\lambda^{i_\ell,\ldots, i_1}_{I_\ell,\ldots, I_1}(R)$.
\end{ThmDef}

In  the definition, we may take $I_1\subseteq \ldots \subseteq I_\ell$ without losing information:   For any $R$-module $M$ satisfying $H^0_I(M)=M$ for some ideal $I$ of $R$, then for every ideal $J$ of $S$, $H^i_{J}(M)=H^i_{I+J}(M)$.
Moreover, $\lambda^{i_\ell,\ldots, i_1}_{I_\ell,\ldots, I_1}(R, L)=\lambda^{i_\ell,\ldots, i_1}_{I_\ell,\ldots,I_2, 0}(R/I_1, L).$

We note that \`Alvarez Montaner has also given a generalization of the Lyubeznik numbers in the characteristic zero setting in \cite{AM-NumInv}; his invariants are defined in terms of $D$-module characteristic cycles. 

{\edit{

In general,  to avoid the dependence on the choice of coefficient field of $\widehat{R}$ in the definition of generalized Lyubeznik numbers, one would need to answer the following question, asked by Lyubeznik.

\begin{question}[\cite{LyuSurveyLC}] \label{LyuQuestion}
Let $S$ be a complete regular local ring of equal characteristic.  For $1 \leq j \leq \ell$, fix $i_j \in \NN$ and ideals $J_j \subseteq S$.
Given any two coefficient fields $K$ and $L$ of $S$, is
$$\Length_{D(S,K)} H^{i_\ell}_{J_\ell}  \cdots  
H^{i_2}_{J_2}H^{n-i_1}_{J_1}(S)=
\Length_{D(S,L)} H^{i_\ell}_{J_\ell}  \cdots  
H^{i_2}_{J_2}H^{n-i_1}_{J_1}(S) \text{?}$$
\end{question} 
\noindent To the best of our knowledge, this question is open even in the case that $s=1$.
}

As their nomenclature indicates, the generalized Lyubeznik numbers are, in fact, generalizations of the Lyubeznik numbers, as the following proposition makes explicit.
\begin{proposition}
If $(R, \fm, K)$ is a local ring containing a field, then $\lambda_{i,j} (R)=\lambda^{i,j}_{\fm,0} (R;L)$
for any coefficient field $L$ of $\widehat{R}.$
\end{proposition}

\begin{proof}[Sketch of proof]  Since the Bass numbers with respect to the maximal ideal are not affected by completion, $R$ is, without loss of generality, complete.  

Consider a surjection $\pi: S:= K \llbracket x_1, \ldots, x_n \rrbracket  \surj R$, where $\pi(K)=L$.  If $\fn$ denotes the maximal ideal of $S$, then $H^i_{\fn} H^{n-j}_{I} (S)$ is injective by \cite[Corollary 3.6]{LyubeznikFinitenessLocalCohomologyModules}.  
Thus, by  \cite[Lemma 1.4]{LyubeznikFinitenessLocalCohomologyModules}, $\dim_K \Hom(K, H^i_{\fn} H^{n-j}_I(S) ) = \dim_K \Ext^i_S(K, H^{n-j}_I(S))$, which, in turn, equals $\lambda_{i,j}(R).$  As $H^i_{\fn} H^{n-j}_{I} (S)$ is supported at $\fm$, it is isomorphic to $E_S(K)^{\oplus \ell}\cong E_S(L)^{\oplus \ell}$ for some $\ell\in\NN$.  Since $E_S(L)$ is a simple $D(S,L)$-module, we have that $\dim_K \Hom(K, H^i_{\fn} H^{n-j}_I(S) ) = \ell = \Length_{D(S,L)} H^i_{\fn} H^{n-j}_I(S) = \lambda_{\fm, 0}^{i,j}(R; L).$ 
\end{proof}

The proof that the generalized Lyubeznik numbers are well defined critically uses the key functor $G$ defined in Subsection \ref{Key Functor}; we now sketch this proof. 

\begin{proof}[Sketch of proof of Theorem/Definition \ref{GenDef}]
{\edit{
First, note that the $D(S,K)$-module length  of the iterated local cohomology modules is finite by \cite{BjorkRingsDifferentialOperators,LyubeznikFinitenessLocalCohomologyModules,LyubeznikFModulesApplicationsToLocalCohomology}.
}}
Again, without loss of generality, we may assume that $R$ is complete.
Along with the surjection $\pi: S=K \llbracket x_1, \ldots, x_n \rrbracket  \surj R$, 
fix $S' = K \llbracket y_1, \ldots, y_{n'} \rrbracket$ and 
take a surjection $\pi' : S' \surj R$ satisfying $\pi'(K) = L$.
Set $I' = \Ker(\pi')$, and let $\fm'$ be the maximal ideal of $S'$.
 Let ${J_{j}}' = {(\pi')}^{-1}(I_j)$ for $1 \leq j \leq s$.
 
Let $S''=K \llbracket x_1,\ldots,x_n, y_1, \ldots, y_{n'} \rrbracket $ and 
let $\pi''$ denote the surjection $\pi'': S''\surj R$ defined by
\begin{alignat*}{3}
&x_i \mapsto \pi(x_i) & \ \ & \text{for } 1\leq i\leq n, \text{ and} \\
&y_j \mapsto \pi'(y_j) & & \text{for } 1\leq j\leq n'.
\end{alignat*}
Let ${J_j}'' = (\pi'')^{-1}(I_j)$ for each $1 \leq j \leq s$. 

Since $\pi$ is surjective, there exist $\sigma_j \in S$ such that $\pi(\sigma_j) = \pi''(y_j)$ for each $1 \leq j \leq n'$.  
Let $\varphi: S'' \to S$ denote the map defined by $\varphi(x_i) = x_i$ for $1 \leq i \leq n$ and $\varphi(y_j) = \sigma_j$ for $1 \leq j \leq n'$, so that
$\varphi$ splits the inclusion $S \inj S''$.
Since for $1 \leq j \leq n'$, $y_1 - \sigma_1 \in \Ker \varphi$, and the map $S''/(y_1 - \sigma_1, \ldots, y_{n'}-\sigma_{n'}) \to S$ induced by $\varphi$ is an isomorphism, $\Ker(\varphi) = (y_1 - \sigma_1, \ldots, y_{n'}-\sigma_{n'}).$
Thus, ${J_j}'' = J_j + (y_1 - \sigma_1, \ldots, y_{n'}-\sigma_{n'})$.  Moreover, the elements $x_1, \ldots, x_n, y_1 - \sigma_1, \ldots, y_{n'}-\sigma_{n'}$ form a regular system of parameters for $S''$, by Properties \ref{FunctorProperties} {\edit{($6$) and ($2'$)}}
of the functor $G$ defined in Subsection \ref{Key Functor},
we obtain that
\[  \Length_{D(S,{K})} H^{i_\ell}_{J_\ell}  \cdots  H^{i_2}_{J_2}  H^{n-i_1}_{J_1}(S) = \Length_{D(S'',{K})} H^{i_\ell}_{J_\ell''}   \cdots H^{i_2}_{J_2''} H^{n'+n-i_1}_{J_1''}(S''),
\] which, by an analogous argument, equals $\Length_{D(S',{K})} H^{i_\ell}_{J_\ell'}  \cdots  H^{i_2}_{J_2'}  H^{n'-i_1}_{J_1'}(S')$ as well.
\end{proof}

Some vanishing properties of the generalized Lyubeznik numbers are as follows; \emph{cf.} Properties \ref{FirstProp}.

\begin{properties}
Given any local ring $(R, \fm, K)$ containing a field, ideals $I_1 \subseteq \ldots, \subseteq I_\ell$ of $R$, $i_j \in \NN$ for $1 \leq j \leq \ell$, and a coefficient field $L$ of $\widehat{R}$, the following hold.
\begin{enumerate}
\item $\lambda^{i_1}_{I_1}(R; L)\neq 0$ if $i_1=\dim(R/I_1)$.
\item  $\lambda^{i_2,i_1}_{I_2,I_1}(R; L)=0$ if $i_2>i_1$, and is nonzero if $i_1=\dim(R/I_1)$ and $i_2=\dim(R/I_1)-\dim(R/I_2)$.
\item  $\lambda^{i_\ell,\ldots, i_1}_{I_\ell,\ldots, I_1}(R; L)=0$ if either $i_1>\dim(R/I_1)$, or if $i_j>\dim(R/I_{j-1})$ and $2 \leq j\leq \ell$.
\end{enumerate}
\end{properties}

The generalized Lyubeznik numbers behave in the following way under finite field extensions: 
if $K\subseteq K'$ is a finite field extension, $R=K \llbracket x_1,\ldots,x_n \rrbracket ,$ and $R'=K' \llbracket x_1,\ldots,x_n \rrbracket ,$ then for all ideals $I_j$ of $R$ and $i_j \in \NN$, $1 \leq j \leq \ell$, $\lambda^{i_\ell,\ldots, i_1}_{I_\ell,\ldots, I_1}(R)= \lambda^{i_\ell,\ldots, i_1}_{I_\ell R',\ldots, I_1 R'}(R').$

Although the Lyubeznik numbers cannot distinguish between complete intersection rings (see Properties \ref{FirstProp}), or between one-dimensional rings, the generalized Lyubeznik numbers can, as the following properties show.

\begin{properties} \label{exsGen} The following hold.  
\begin{enumerate}
\item Fix a complete local ring $(R, \fm, K)$ of dimension one and a coefficient field $L$ of $R$, and let $P_1,\ldots P_t$ denote the minimal primes of $R$.  Then 
$\lambda^{1}_{0} (R; L)=\lambda^{1}_{0} (R/P_1; L)+\ldots +\lambda^{1}_{0} (R/P_t ; L)+t-1.$
\item If $K$ is a field and $S=K \llbracket x_1,\ldots, x_n \rrbracket $, 
let $f=f^{\alpha_1}_1\cdots f^{\alpha_t}_t$, where $f_1,\ldots ,f_t\in S$ are irreducible and each $\alpha_i \in \NN$. Then
$\lambda^{n-1}_{0} (S/f)\geq\lambda^{n-1}_{0} (S/f_1)+\ldots +\lambda^{n-1}_{0} (S/f_t)+t-1.$
\end{enumerate}
\end{properties}

Proposition \ref{exsGen} indicates that the generalized Lyubeznik numbers of the form $\lambda_0^j(R; L)$ are of specific interest.  One way they can be used is to define the following invariant of a local ring containing a field.

\begin{definition} [Lyubeznik characteristic]
Let $R$ be a $d$-dimensional local ring containing a field, and fix  a coefficient field $L$ of $\widehat{R}$.  The \emph{Lyubeznik characteristic of $R$ with respect to $L$} is defined as \[ \chi_{\lambda}(R; L) := \sum_{i=0}^d (-1)^i \lambda_{0}^i(R; L).\]  
\end{definition}

As a consequence of the Mayer-Vietoris sequence for local cohomology (see Remark \ref{MVSequence}), 
given ideals $I$ and $J$ of a local ring $R$ containing a field, 
$\chi_{\lambda}(R/I; L) + \chi_{\lambda}(R/J;  L) = \chi_{\lambda}(R/(I+J); L) + \chi_{\lambda}(R/I\cap J; L)$
for any coefficient field $L$ of $\widehat{R}$.
Note that the Lyubeznik characteristic of a Stanley-Reinser ring is computed in Theorem \ref{LyuCharSR}.

\edit{
We now discuss some connections of certain generalized Lyubeznik numbers with singularities in  positive characteristic.
Suppose that $R$ is a ring of characteristic $p>0$.
In this case, we have the Frobenius endomorphism $F:R\to R$ defined by $F(r)= r^p$, as discussed in Section \ref{ProjSchemesSec}.
A variety of types of singularities can be defined via Frobenius (\emph{e.g.}, $F$-\emph{injective}, $F$-\emph{pure}, and $F$-\emph{regular} singularities).

If $R$ is reduced, let $R^{1/p^e}=\{r^{1/p^e}\mid r\in R\}$, the ring obtained by adjoining the $p^e$-th roots of elements of $R$.
We say that $R$ is \emph{$F$-finite} if $R^{1/p}$ is a finitely generated $R$-module. 
Throughout this discussion, we assume that $R$ is reduced and $F$-finite.

We now recall some basic definitions from the theory of tight closure developed by Hochster and Huneke \cite{HochsterHunekeTC1,HochsterHunekeTightClosureAndStrongFRegularity};
we also refer to \cite{HunekeTightClosureBook} for a reference. 
If $I$ is an ideal of $R$, the \emph{tight closure of $I$}, $I^*$, is the ideal of $R$
consisting of all elements $z\in R$ for which there exists some $c \in R$
that is not in any minimal prime of $R,$ such that
$$
cz^q \in I^{[q]} \hbox{ for all }q = p^e \gg 0.
$$

An ideal $I$ is \emph{tightly closed} if $I^* = I$.
A ring  is called  \emph{$F$-rational} if every parameter ideal is tightly closed, and is called \emph{weakly  $F$-regular} if every ideal of the ring is tightly closed.
A ring is called \emph{$F$-regular} if every localization of the ring is weakly $F$-regular. 
It is not known whether every weakly $F$-regular ring is $F$-regular. 
We point out that tight closure does not commute with localization in general \cite{Brenner-Monsky}.

The \emph{test ideal} of $R$ is defined by 
$$
\tau(R)= \bigcap_{\substack{J \subseteq  R \text{ ideal} }} (J:J^*).
$$ 
The test ideal of $\tau(R)$ is an important object in understanding how ``far" a ring is from being weakly $F$-regular. In particular,  
$\tau(R)=R$ if and only if $R$ is weakly $F$-regular.

If the $R$-module inclusion $R\hookrightarrow R^{1/p}$ splits, $R$ is called \emph{$F$-pure}. The $F$-purity of a ring simplifies computations for cohomology groups and implies vanishing properties of these groups \cite{HochsterRobertsFrobeniusLocalCohomology,LyubeznikCD}. 
If for every $c\in R$ not contained in any minimal prime of $R$, there exists $e\geq 1$ such that the $R$-module map $R \to R^{1/p^e}$ defined by $1\mapsto c^{1/p^e}$ splits, then $R$ is called \emph{strongly $F$-regular} . 
A local ring $(R,\mathfrak{m},K)$ is \emph{$F$-injective} if the action of Frobenius on the local cohomology module $H^i_{\mathfrak{m}}(R)$ is injective for every $i\in\NN$. 
The relations among these properties are the following:
$$
\begin{Bmatrix}
F\hbox{-injective}\\
\hbox{rings}\\
\end{Bmatrix}
\supseteq
\begin{Bmatrix}
F\hbox{-pure}\\
\hbox{rings}\\
\end{Bmatrix}
\supseteq
\begin{Bmatrix}
\hbox{weakly}\\
F\hbox{-regular}\\
\hbox{rings}\\
\end{Bmatrix}
\supseteq
\begin{Bmatrix}
F\hbox{-regular}\\
\hbox{rings}\\
\end{Bmatrix}
\supseteq
\begin{Bmatrix}
\hbox{strongly}\\
F\hbox{-regular}\\
\hbox{rings}\\
\end{Bmatrix}
\supseteq
\begin{Bmatrix}
\hbox{regular}\\
\hbox{rings}\\
\end{Bmatrix}.
$$
}

The generalized Lyubeznik numbers of the form $\lambda_0^j(R)$ can 
provide information on singularities in characteristic $p>0$.
In particular, results of Blickle indicate that they encode information on $F$-regularity and $F$-rationality.
\edit{ We present these results, as done in \cite{NuWiEqual}}.

\edit{
\begin{theorem} \label{BliThm} \cite{BlickleIntHomology}
Suppose that $(R, \fm, K)$ is a $d$-dimensional complete local domain of characteristic $p>0$.  
If $R$ is $F$-injective and $\lambda_0^d(R;L)=1$, then $R$ is $F$-rational.  
Moreover, if $K$ is perfect and $R$ is $F$-rational, then $\lambda_0^d(R)=1$.
\end{theorem}
}
 

{\edit{
\begin{example} Suppose that $K$ is a field, and let $R = K[X]$ be the polynomial ring over $K$ in the entries of an $r \times r$ square matrix $X$ of indeterminates, $r > 1$.
Let $\fm$ denote the homogeneous maximal ideal of $R$, and let $\det X$ denote the determinant of $X$.
When $K$ is a perfect field of characteristic $p>0$, $R_\fm/(\det X) R_\fm$ is strongly $F$-regular, so also $F$-rational, and $\lambda_0^{r^2-1} \left (R_\fm/(\det X) R_\fm\right)=1$ \cite{HochsterHunekeTightClosureOfParameterIdealsAndSplitting}.
However, when $K$ has characteristic zero, we have that  $\lambda_0^{r^2-1} \left(R_\fm/(\det X) R_\fm\right)\geq 2$ \cite{WaltherBS}.
\end{example}
}

Work in \cite{NunezPerezFJacobian} indicates that the generalized Lyubeznik numbers can measure how ``far" an $F$-pure hypersurface ring is from being $F$-regular.  Take an element $f$ of an $F$-finite regular local ring $R$ of characteristic $p>0$ such that $R/(f)$ is $F$-pure.  
Let $\tau_1$ be the pullback of the test ideal of $R/(f)$ to $R$, and for $i>1$, and let $\tau_i$ denote the pullback of the test ideal of $R/\tau_{i-1}$ to $R.$ Vassilev showed that the chain $0 \subsetneq \tau_1 \subsetneq \tau_2 \subsetneq \cdots \subsetneq \tau_s = R$ is finite and each $R/\tau_i$ is $F$-pure \cite{Vassilev}.  In fact, $\lambda_0^{\dim\left(R/(f) \right)} \left(R/(f); L\right) \geq s$ for any coefficient field $L$ of $\widehat{R/(f)}$.
In particular, if $\lambda_0^{\dim\left(R/(f) \right)} \left(R/(f); L\right) =1$, then $R/(f)$ is $F$-regular, 
which recovers the result of Blickle in this case (\emph{cf.} Theorem \ref{BliThm}).  

{\edit{
\begin{question} 
What further geometric and algebraic properties are captured by the generalized Lyubeznik numbers of a local ring
containing a field? In particular, do these numbers measure
singularities in characteristic zero? 
\end{question}  
}}
\section{Lyubeznik numbers of Stanley-Reisner rings}

\begin{definition}[Simplicial complex, face/simplex, dimension of a simplicial complex or face, facet]
A \emph{simplicial complex} $\Delta$ on the vertex set
$[n] = \{1, \ldots, n\}$ is a collection of subsets, called \emph{faces} or \emph{simplices}, that is closed under
taking subsets. The \emph{dimension} of a face
$\sigma \in\Delta$ is defined as $\dim(\sigma)=|\sigma|-1$. The \emph{dimension}  of $\Delta$, $\dim(\Delta),$ is the maximum of the dimensions of
its faces; by convention, $\dim(\Delta) := -1$ if $\Delta=\varnothing$. We let $F_i(\Delta)$ denote the set of faces of $\Delta$ dimension $i$.
A \emph{facet} of $\Delta$ is a face of $\Delta$ that is not strictly contained in any other face.

\end{definition}
\edit{If $\Delta_1$ and $\Delta_2$ are simplicial complexes on the vertex set $[n]$, then
$\Delta_1\cap\Delta_2$ and $\Delta_1\cup\Delta_2$ are also simplicial complexes.
Moreover, a simplicial complex
is determined by its facets. } 

\begin{definition}[Alexander dual of a simplicial complex, link]
The Alexander dual  of a simplicial complex $\Delta$ on $[n]$ is defined as 
\[ \Delta^\vee: =\{[n]\setminus \sigma \mid \sigma\in\Delta\}.\]
Given a simplex $\sigma\in \Delta,$
 we define its link inside $\Delta$ by 
$\link_\Delta=\{\tau\in\Delta\mid \tau\cup\sigma\in\Delta\hbox{ and }\tau\cap\sigma=\varnothing \} $
\end{definition}

\begin{notation}
Let $\Delta$ be a simplicial complex on the vertex set $[n]$ and $\sigma\in\Delta.$ Then $x^\sigma$ denotes $\prod \limits_{i\in\sigma} x_i\in S=K[x_1,\ldots,x_n].$
\end{notation}
\begin{definition}[Stanley-Reisner ideal and ring]  Given a field $K$, the \emph{Stanley-Reisner ideal associated to the simplicial complex $\Delta$} of $K[x_1,\ldots,x_n]$
is the squarefree monomial ideal
$I_{\Delta} = \langle x^\sigma \mid \sigma\not\in\Delta \rangle$. The \emph{Stanley-Reisner ring of $\Delta$} is $K[x_1,\ldots,x_n]/I_\Delta$.
\end{definition}

\begin{theorem}[see \textup{\cite[Theorem 1.7]{MillerSturmfels}}]\label{Bijection}
Given a field $K$, 
the map sending $\Delta$ to the ideal $I_\Delta$ of $S = K[x_1, \ldots, x_n]$ defines a bijection from
simplicial complexes on the vertex set $[n]$ to squarefree monomial ideals of $S$. 
\end{theorem}

\begin{example}\label{ExMon2}
Given a field $K$, let $S=K[x_1,\ldots,x_6]$ and let
\[I = (x_1 x_2 x_3, x_1 x_6, x_2 x_4, x_2 x_5, x_2 x_6, x_3 x_6).\]  Then $I = I_\Delta$, where $\Delta$ is the following simplicial complex:

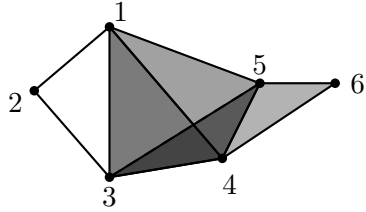
\begin{center}
\begin{pspicture}(3, 3)
\rput(1.15,2.7){$1$}
\rput(-0.25,1.5){$2$}
\rput(1,0.25){$3$}
\rput(2.6,0.4){$4$}
\rput(4.3,1.75){$6$}
\rput(3,2){$5$}
\psline(1,2.5)(0,1.65)(1,0.5)
\pscustom[fillstyle=solid, fillcolor=black, opacity=0.45]{\psline[linestyle=none](1,0.5)(3,1.75)(2.5,0.75)(1,0.5)}
\pscustom[fillstyle=solid, fillcolor=black, opacity=0.52]{\psline(1,0.5)(1,2.5)(2.5,0.75)(1,0.5)}
\pscustom[fillstyle=solid, fillcolor=black, opacity=0.3]{\psline(2.5,0.75)(4,1.75)(3, 1.75)}
\pscustom[fillstyle=solid, fillcolor=black, opacity=0.37]{\psline(1,2.5)(3,1.75)(2.5,0.75)(1,2.5)}
\psline[linestyle=dashed](1, 0.5)(3, 1.75)
\psdot(0,1.65)
\psdot(1, 2.5)
\psdot(1,0.5)
\psdot(2.5,0.75)
\psdot(3,1.75)
\psdot(4, 1.75)
\end{pspicture}
\end{center}
\end{example}

We now fix notation for the following discussion.

\begin{notation}
Suppose that $K$ is a field, let $S=K[x_1,\ldots,x_n]$, and let $\fm$ denote its homogeneous maximal ideal $(x_1, \ldots, x_n)$.   
Fix the standard grading on $S$; \emph{i.e.}, $\deg(x_i)=1$ for $1\leq i\leq n$.
Moreover, fix a simplicial complex $\Delta$ on vertex set $[n]$, and set $I=I_\Delta.$
\end{notation}

Suppose that $I$ is an squarefree monomial ideal in $S$, and let $R=S/I.$
Since $I$ is a homogeneous ideal, $H^{n-j}_I(S)$  and $H^i_\fm H^{n-j}_I(S)$ are graded $S$-modules for all $i, j \in \NN$ (see Remark \ref{LocalDuality}). Results of 
Lyubeznik imply that $H^i_\fm H^{n-j}_I(S)$ is isomorphic, as $S$-modules, with a direct sum of $\lambda_{i,j}(R_\fm)$ copies the injective hull $E_S(K)$ \cite{LyubeznikInjectiveDim}.
Results of Y.\ Zhang in characteristic $p>0$, and extended to arbitrary characteristic by Ma and the third author, imply that,
in fact, $H^i_\fm H^{n-j}_I(S)\cong H^n_\fm(S)^{\oplus \lambda_{i,j}(R_\fm)}$ as graded $S$-modules  \cite{EulerianDModules,YiZhang}. 
This was previously known for monomial ideals (see, for instance, \cite{MustataSRrings,Yanagawa}).
In particular, 
$\lambda_{i,j}(R_\fm)$ is equal to the $K$-vector space dimension of the $(-n)$-degree piece of $H^i_\fm H^{n-j}_I(S).$

Musta{\c{t}}{\v{a}} used the $\NN^n$-graded structures of $S$ and of $H^j_I(S)$ to describe the structure of  $H^j_I(S)$ in terms of $\Delta$: 
for every $\alpha\in \ZZ^n,$ $\left[H^j_I(S)\right]_{-\alpha}=
\widetilde{H}^{j-2}(\Delta^\vee;K)$ \cite{MustataSRrings}.
In addition, in \cite{Yanagawa},  Yanagawa  used this graded structure on the local cohomology modules  to prove that 
$$
\lambda_{i,j}(R)=\dim_K \left[\Ext^{n-i}_S\(\Ext^{n-j}_S(S/I,S)\),S\right]_{(0,\ldots,0)}.
$$
{\edit{We point out that there exist analogous formulas for the Bass
numbers of local cohomology modules of simplicial normal Gorenstein semigroup rings with support in squarefree monomial ideals \cite{YanagawaSheaves}.}}

Let $F_\bullet$ denote the minimal graded free resolution of $I_{\Delta^\vee}$,
$$
F_\bullet : \ \ \ 0\to F_{\ell}\FDer{d_\ell}F_{\ell}\overset{d_{\ell-1}}{\longrightarrow} \ldots \FDer{} F_1\FDer{d_1} F_0 \to I_{\Delta^\vee} \to 0,
$$ 
where $d_i(F_{i})\subseteq \fm F_{i-1}$. 
Since the resolution is graded, we write $F_i$ as a direct sum of copies of $S$ with degree shifts,
$$
F_i=\bigoplus_{j\in\ZZ} S(-j)^{\beta_{i,j}},
$$
where $\beta_{i,j}$ is called the \emph{$(i,j)$-th graded Betti number,}  an invariant of $I_{\Delta^\vee}.$
\edit{We define the \emph{$r$-linear strand of $I_{\Delta^\vee}$}  by
$$
F^{\langle r \rangle}_\bullet: \ \ \ 0\to F^{\langle r \rangle}_{n-r}\FDer{d_{n-r}}F^{\langle r \rangle}_{n-r-1}\overset{d_{n-r-1}}{\longrightarrow} \ldots \FDer{} F^{\langle r \rangle}_1\FDer{d_1} F^{\langle r \rangle}_0\to 0, 
$$ 
where
$$
F^{\langle r \rangle}_i=\bigoplus_{|j|=i+r} S(-j)^{\beta_{i, j}}
$$
and the maps correspond to the components in the resolution $F_\bullet$.
We consider the complex $G^{\langle r \rangle}_\bullet$ given by tensoring $F^{\langle r \rangle}_{\bullet}$ with $\left. S \middle/ (x_1-1,\ldots,x_n-1)\right.$ and taking the corresponding monomial matrices (see \cite{MillerE}).
}
\begin{theorem}[\cite{AMV}]\label{ThmAMV}
Let $K$ be a field and  let $S=K[x_1,\ldots,x_n]$.  Fix $\Delta$, a simplicial complex on vertex set $[n]$, and let $I=I_\Delta$  denote the Stanley-Reisner ideal of $S$ associated to $\Delta$; moreover, let $R=S/I_\Delta$ be the corresponding Stanley-Reisner ring.
Let $I_{\Delta^\vee}$ be the Stanley-Reisner ideal associated to the Alexander dual, $\Delta^\vee,$ of $\Delta$, and\edit{ let $G^{\langle r \rangle}_\bullet$ be the complex defined in the previous paragraph.

Then
$$
\lambda_{i,j}(R)=\dim_K H_{j-i} \left (G^{\langle n-j\rangle}_\bullet\right),
$$ where $H_\bullet$ denotes homology.}
\end{theorem}
\edit{
\begin{remark}\label{refRem}
The form of Theorem \ref{ThmAMV} is not exactly as it appears in \cite{AMV}.  Their result transposes the monomial matrices of $G^{\langle r\rangle}_{\bullet}$, and uses cohomology because the category of hypercubes that they consider is contravariant.
\end{remark}
}
{\edit{
Since there are algorithms for computing  $F_\bullet,$ 
and so $G^{\langle n-j\rangle}_\bullet,$ the previous theorem gives an algorithm to compute
the Lyubeznik numbers of Stanley-Reisner rings in any characteristic.}} There are other algorithms and techniques for computing Lyubeznik numbers for such rings \cite{AM-CharCyclesI,AM-CharCyclesII,Yanagawa}. 

The algorithm given by Theorem \ref{ThmAMV} was implemented by Fern\'andez-Ramos \cite{Tutorial} in 
{\tt Macaulay2} \cite{Macaulay} (see \cite{MONICA} for details). 
\begin{example}
For $S$, $I$,  and  $\Delta$ from Example \ref{ExMon2},
$$
\Lambda(S/I) =\begin{pmatrix}
 0 & 0 & 0 & 0 & 0 \\
 0 & 0 & 0 & 0 & 0 \\
 0 & 0 & 0 & 0 & 0 \\ 
 0 & 0 & 0 & 0 & 0 \\
 0 & 0 & 0 & 0 & 1 \\
\end{pmatrix}.
$$
\end{example}
\begin{example}
Let $K$ be a field, let $S = K[x_1, \ldots, x_6]$, and let $I$ be as defined in \eqref{exmon1}.  Then 
$$
\Lambda(S/I) =\begin{pmatrix}
   0 & 0 & 0  & 0  \\
   0 & 0 & 0  & 0  \\
   0 & 0 &  0 & 0  \\
   0 & 0  & 0 & 1 \\
\end{pmatrix}.
$$
\end{example}

\begin{example}[\cite{AMV}]
\label{LyuNum S-R differ}
Let $S=K[x_1,\ldots,x_6]$ and let $I$ denote the monomial ideal of $S$ generated by
\begin{alignat*}{20}
 &x_1x_2x_3 , \ \ &x_1x_2x_4 , \ \ &x_1x_3x_5, \ \ &x_1x_4x_6, \ \ &x_1x_5x_6, \ \
&  &x_2x_3x_6, \ \ &x_2x_4x_5, \ \ &x_2x_5x_6, \ \ &x_3x_4x_5, \ &\text{ and } \ &x_3x_4x_6.
\end{alignat*}
The simplicial complex associated to $I$ corresponds to a minimal triangulation of $\mathbb{P}^{2}_\mathbb{R}$, and the projective algebraic set that $I$ defines in $\mathbb{P}^{5}_K$ has been called \emph{Reisner's variety} since he introduced it in 
\cite[Remark 3]{Reisner}.  If $K=\QQ$, then
$$
\Lambda(S/I) =\begin{pmatrix}
 0 &  0 & 0 & 0 \\
 0 &  0 & 0 & 0 \\
 0 &  0 & 0 & 0 \\
 0 &  0 & 0 & 1 \\
\end{pmatrix},
$$
but if $K=\ZZ/2\ZZ,$ then
$$
 \Lambda(S/I) = \begin{pmatrix}
 0 & 0 & 1 & 0 \\
 0 & 0 & 0 & 0 \\
 0 & 0 & 0 & 1 \\
 0 & 0 & 0 & 1 \\
\end{pmatrix}.
$$
This example shows that the Lyubeznik numbers of Stanley-Reisner rings depend on the characteristic. {\edit{We point out that Blickle has found examples for which these invariants exhibit ``bad" behavior under reduction to positive characteristic \cite{BlickleLyuCIS}.}}
\end{example}

The generalized Lyubeznik numbers for Stanley-Reisner rings can also be computed  as certain lengths in the category of straight modules
\cite{Yanagawa}, and in terms of characteristic cycle multiplicities in characteristic zero \cite{NuWiEqual}.
Indeed, let $K$ be a field and let $S=K[x_1,\ldots,x_n]$, so that its completion at its homogeneous maximal ideal is $\widehat{S}=K \llbracket x_1,\ldots,x_n \rrbracket $.
Let $I_1,\ldots,I_\ell$ be  ideals of $S$ generated by squarefree monomials, and fix $i_1, \ldots, i_\ell \in \NN$.
Then 
\begin{align*} \lambda^{i_\ell,\ldots, i_1}_{I_\ell,\ldots,I_1} \( \widehat{S}\)
=\Length_{{\bf Str}} 
H^{i_\ell}_{I_\ell} \cdots H^{i_2}_{I_2} H^{n- i_1}_{I_1} (\omega_S) 
= \sum_{\alpha \in \{0,1\}^n} \dim_k \left[ H^{i_\ell}_{I_\ell} \cdots H^{i_2}_{I_2} H^{n-i_1}_{I_1} (\omega_S)\right]_{-\alpha}.
\end{align*}
Moreover, if $K$ has characteristic zero, then
$
\lambda^{i_1,\ldots, i_\ell}_{I_1,\ldots,I_\ell} \(\widehat{S}\)
=\Mult \(H^{i_\ell}_{I_\ell} \cdots H^{i_2}_{I_2} H^{n-i_1}_{I_1} (S)\),
$ where $\Mult(M)$ denotes the characteristic cycle multiplicity of a $D(S,K)$-module $M$.  

Therefore, certain generalized Lyubeznik numbers of a characteristic zero Stanley-Reisner ring can be computed using algorithms for calculating characteristic cycle multiplicities (see \cite{AM-CharCyclesI,AM-CharCyclesII,AM-GL-Z}).
In addition, one can also compute the Lyubeznik characteristic using these algorithms. \edit{We point out that the generalized Lyubeznik numbers may differ from the characteristic cycle multiplicities if the ideals are not monomial.}

\begin{example}\label{GenLyuNum S-R differ}
Take $S$ and $I$ from Example \ref{LyuNum S-R differ}, and let $R=S/I.$
If $K=\QQ$, then $\lambda^4_0(R)=31$ and all other $\lambda^j_0(R)$ vanish.
If $K=\ZZ/2\ZZ$, then $\lambda^4_0(R)=32$, $\lambda^3_0(R)=1$, and all other $\lambda^j_0(R)=0$.
Notice that in both cases,
$
\chi_\lambda(R)=31.
$
\end{example}

Unlike the Lyubeznik numbers and the generalized Lyubeznik numbers (see Examples \ref{LyuNum S-R differ} and \ref{GenLyuNum S-R differ}), the Lyubeznik characteristic is characteristic-independent in this setting, as the following result shows.

\begin{theorem}[\cite{NuWiEqual}]\label{LyuCharSR}
Take a simplicial complex $\Delta$ on vertex set $[n]$. Let $R$ denote the Stanley-Reisner ring of $\Delta$, and let $\fm$ be its homogeneous maximal ideal.
Then
$$
\chi_\lambda(R_\fm)=\sum^{n}_{i=-1} (-2)^{i+1} |F_i (\Delta)|.
$$
\end{theorem}

\begin{example}
Take $S$ and $\Delta$ from Example \ref{ExMon2}, let $R = S/I_\Delta$, and let $\fm$ denote the homogeneous maximal ideal of $R$. Then by Theorem \ref{LyuCharSR}, 
\begin{alignat*}{3}
\chi_\lambda(R_\fm) &= |F_{-1} (\Delta)|+ (-2) |F_0 (\Delta)| + 4 |F_1 (\Delta)| - 8 |F_2 (\Delta)| + 16 |F_3 (\Delta)| \\
&= 1-2 \cdot 6  + 4 \cdot 10  - 8 \cdot 5  + 16 \cdot 1=  5.
\end{alignat*}
In this case,
$
\lambda^4_0(R_\fm)=6, \ \lambda^3_0(R_\fm)=2, \ \lambda^2_0(R_\fm)=1,
$
and all other $\lambda^\bullet_0(R_\fm) =0$.
On the other hand, we verify that
$$
\chi_\lambda(R_\fm)=\lambda^4_0(R_\fm)-\lambda^3_0(R_\fm)+\lambda^2(R_\fm)=6-2+1=5.
$$
\end{example}

\section{Lyubeznik numbers in mixed characteristic}

Although the Lyubeznik numbers are defined only for local rings containing a field (\emph{i.e.}, equal characteristic), 
in \cite{NunezBWittLyuNumMixed}, our reference for this section, an alternate definition of Lyubeznik numbers is given for all local rings with characteristic $p>0$ residue field.
These invariants are called the \emph{Lyubeznik numbers in mixed characteristic}, as they are, in particular, defined for local rings of mixed characteristic.  

The definition of the Lyubeznik numbers in mixed characteristic, like the Lyubeznik numbers, relies on the Cohen Structure Theorems.  
\edit{For any field $K$ of characteristic $p>0$, there is a unique (up to isomorphism) complete (unramified) Noetherian discrete valuation ring of mixed characteristic of the form $(V, pV, K)$.
Any complete local ring $(R, \fm, K)$ of mixed characteristic $p>0$ is the homomorphic image of some $V \llbracket x_1, \ldots, x_n \rrbracket $, where $(V, pV, K)$ is the complete Noetherian DVR corresponding to $K.$}
As any complete local ring $(R, \fm, K)$ of equal characteristic $p>0$ admits a surjection of some $K \llbracket x_1, \ldots, x_n \rrbracket $, the natural surjection $V \surj K$ induces the surjective composition $V \llbracket x_1, \ldots, x_n \rrbracket  \surj K \llbracket x_1, \ldots, x_n \rrbracket  \surj R$.

The fact that the Lyubeznik numbers in mixed characteristic are well defined relies not only on the existence of a surjection from a regular local ring, but on details of its construction \cite{Cohen}. 
These invariants are defined as follows.

\begin{ThmDef}
Let $(R,\fm,K)$ be a local ring such that $K$ has characteristic $p>0$.
By the Cohen Structure Theorems, there exists a surjection $\pi:S\surj \widehat{R}$, where $(S, \fn, K)$ is an $n$-dimensional unramified regular local ring of mixed characteristic.  
Let $I = \Ker(\pi)$, and take $i, j\in \NN$.
The \emph{Lyubeznik number of $R$ in mixed characteristic with respect to $i$ and $j$} is defined as 
$$
\widetilde{\lambda}_{i,j}(R):=\Dim_K \Ext^i_S\left(K,H^{n-j}_I (S)\right).
$$
This number depends only on $R$, $i$, and $j$; \emph{i.e.}, it is independent of the choice of $S$ and of $\pi$.  If $(R, \fm, K)$ is \emph{any} local ring such that $K$ has characteristic $p>0$, then $\lambda_{i,j} (R) := \lambda_{i,j} (\widehat{R}),$ where $\widehat{R}$ is the completion of $R$ with respect to $\fm$. 
\end{ThmDef}

\begin{proof}[Sketch of proof]  
We know that each $\Dim_K \Ext^i_S\left(K,H^{n-j}_I (S)\right)$ is finite by \cite{LyuUMC, Nunez}. Without loss of generality, we may assume that $R$ is complete.   We will prove that the invariants are well defined in several steps.  

 \emph{Step 1.}  Take a coefficient ring $W$ of $R$, and take a complete Noetherian DVR $V$ with residue field $K$.  Take surjections $\pi: S := V \llbracket x_1, \ldots, x_{n-1} \rrbracket  \surj R$ and $\pi' : S' := V \llbracket y_1, \ldots, y_{n'-1} \rrbracket  \surj W$ such that $\pi(V) = W = \pi'(V)$, and $\pi(v) = \pi'(v)$ for all $v \in V$.  
 Let $I = \Ker(\pi)$ and let $I' = \Ker(\pi')$.
 We then have a surjection $\pi'' : S'' := V \llbracket x_1, \ldots, x_{n-1}, y_1, \ldots, y_{n'-1} \rrbracket  \surj R$ given by $x_i \mapsto \pi(x_i)$ and $y_j \mapsto \pi'(y_j).$  If $I''= \ker(\pi'')$.  Then similar to the proof of Theorem/Definition \ref{LyuDef}, Properties \ref{FunctorProperties} {\edit{(5) and (7)}} will imply that  
\begin{equation}
\dim_K \Ext_S(K, H^{n-j}_I(S))  = \dim_K \Ext_{S''}(K, H^{n+n'-j}_{I''}(S'')) = \dim_K \Ext_{S'}(K, H^{n'-j}_{I'}(S')).  \label{EasyCase}
\end{equation}

\emph{Step 2.}  Now consider the general case.  Take $V$ and $V'$ complete Noetherian domains with residue field $K$, and surjections $\pi : V \llbracket x_1, \ldots, x_{n-1} \rrbracket \surj R$ and $\pi' : V' \llbracket x_1, \ldots, x_{n'-1} \rrbracket \surj R$ such that $\pi|(V) =  W= \pi'(V')$.  Let $I = \ker(\pi)$ and $I' = \ker(\pi').$
Now, let $\mu = \dim_K(\fm/\fm^2)$, and let $T = V \llbracket x_1, \ldots, x_{\mu-1} \rrbracket $ and $T' = V' \llbracket y_1, \ldots, y_{\mu-1} \rrbracket .$  
It can be shown that there exist surjections $\eta : T \surj R$ and $\eta' : T' \surj R$ and an isomorphism $\varphi : T \to T'$ such that $\eta' \circ \varphi = \eta$.  Let $J = \ker(\eta)$ and $J' = \ker(\eta')$. Then by \eqref{EasyCase},
\begin{alignat*}{3}
&\dim_K \Ext^i_S(K, H^{n-j}_I(S)) & \ = \ & \dim_K \Ext^i_T(K, H^{\mu-j}_J(T)),  \text{ and} \\
&\dim_K \Ext^i_{S'}(K, H^{n'-j}_{I'}(S')) &\ =\ & \dim_K \Ext^i_{T'}(K, H^{\mu-j}_{J'}(T')).
\end{alignat*}
As $\varphi$ is an isomorphism, $\dim_K \Ext^i_T(K, H^{\mu-j}_J(T)) = \dim_K \Ext^i_{T'}(K, H^{\mu-j}_{J'}(T'))$, and we are done.
\end{proof}

In general, the Lyubeznik numbers in mixed characteristic are not easy to compute.  However, if $(V,pV,K)$ is a complete DVR of unramified mixed characteristic $p>0$ and $R=V \llbracket x_1, \ldots, x_n \rrbracket $, and $r_1, \ldots, r_\ell \in R$ is a regular sequence, then $\LyuMixed{i}{j}{R/(r_1, \ldots, r_\ell)}=1$ when $i=j=n+1-\ell$, and vanishes otherwise (\emph{cf.} Properties \ref{FirstProp}).

Some vanishing properties of the Lyubeznik numbers in mixed characteristic are as follows (\emph{cf.} Properties \ref{FirstProp}).

\begin{properties} \label{MixedProp} Let $(R, \fm, K)$ be an $d$-dimensional local ring such that $K$ has characteristic $p>0$.  Then
\begin{enumerate}
\item $\LyuMixed{i}{j}{R} = 0$ if either $i>d$, $j>d,$ or $i> j+i$, and
\item $\LyuMixed{d}{d}{R} \neq 0.$
\end{enumerate}
\end{properties}

Given a surjection $\pi: S \surj \widehat{R}$ as in the definition of the Lyubeznik numbers in mixed characteristic, with $I = \ker(\pi)$ and $n = \dim(S)$, 
we have that $H^{n-j}_I(S) = 0$ for $j > \dim(S/I) = d$, and $\InjDim H^{n-j}_I(S) \leq \dim H^{n-j}_I(S)+1 \leq j+1$ by \cite{ZhouHigherDer}, proving Property \ref{MixedProp} (1) for $j>d$ or $i>j+i$.  However, vanishing for $j>d$ is more subtle, requiring a more detailed analysis.  
These results make possible the following definition (\emph{cf.} Definitions \ref{LyubeznikTable} and \ref{Highest}).  

\begin{definition}[Highest Lyubeznik number in mixed characteristic]
Let $(R, \fm, K)$ be an $d$-dimensional local ring such that $K$ has characteristic $p>0$.  Then $\LyuMixed{d}{d}{R}$ is called the \emph{highest Lyubeznik number of $R$ in mixed characteristic}.
\end{definition}

{\edit{
\begin{question}
What topological properties do the Lyubeznik numbers in mixed characteristic capture?
In particular, what information is encoded in the highest Lyubeznik number in mixed characteristic?
\end{question}

\begin{question}[\cite{NunezBWittLyuNumMixed}]
Let $R$ be a Cohen-Macaulay local ring of dimension $d$. Kawasaki showed that $\lambda_{d,d}(R)=1$ if $R$ contains a field \cite{KawasakiHighestLyubeznikNumber}.  If $R$ has mixed characteristic, is $\widetilde{\lambda}_{d,d}(R)=1$?
\end{question}
}}

\begin{definition}[Lyubeznik table in mixed characteristic]
Let $(R, \fm, K)$ be an $d$-dimensional local ring such that $K$ has characteristic $p>0$.
The \emph{Lyubeznik table of $R$ in mixed characteristic} is the $(d+1) \times (d+1)$ matrix 
\edit{
\[ 
\widetilde{\Lambda}(R) := \left( \LyuMixed{i}{j}{R} \right)_{0 \leq i, j \leq d}  =   \begin{pmatrix}
\LyuMixed{0}{0}{R} &  \LyuMixed{0}{1}{R} & \cdots  & \LyuMixed{0}{d-2}{R} & \LyuMixed{0}{d-1}{R} & \LyuMixed{0}{d}{R} \\
\LyuMixed{1}{0}{R}   & \LyuMixed{1}{1}{R} & \cdots & \LyuMixed{1}{d-2}{R} & \LyuMixed{1}{d-1}{R} & \LyuMixed{1}{d}{R} \\
0 &  \LyuMixed{2}{1}{R}   & \cdots & \LyuMixed{2}{d-2}{R} & \LyuMixed{2}{d-1}{R} & \LyuMixed{2}{d}{R} \\
0 & 0 & \ddots & \vdots & \vdots & \vdots\\
\vdots & \vdots & \ddots &\LyuMixed{d-1}{d-2}{R}  & \LyuMixed{d-1}{d-1}{R} & \LyuMixed{d-1}{d}{R} \\
  0 & 0 & \cdots & 0 &  \LyuMixed{d}{d-1}{R}  & \LyuMixed{d}{d}{R} \\
\end{pmatrix}.
\]}\end{definition}
{\edit{
\noindent We note that all entries below the subdiagonal in each Lyubeznik table in mixed characteristic vanish; it is not known whether all entries below the diagonal must vanish.

\begin{question}
For every local ring $R$ with characteristic $p>0$ residue field $K$, is $\LyuMixed{i}{j}{R}=0$ for $i>j$?
\end{question}
}}

When $(R,\fm, K)$ is a local ring of equal characteristic $p>0$, both the Lyubeznik numbers in mixed characteristic and the (original) Lyubeznik numbers are defined.  When $R$ is Cohen Macaulay, or if $\dim(R)\leq 2$, these invariants coincide:  $\lambda_{i,j}(R) = \LyuMixed{i}{j}{R}$ for all $i,j\in \NN$.
However, the Lyubeznik numbers in mixed characteristic do not, in general, agree with the Lyubeznik numbers for rings of equal characteristic $p>0$, as the following example indicates.  

\begin{example}
Let $S$ be the ring and $I$ the ideal from Example \ref{LyuNum S-R differ}, and let $K = \mathbb{F}_2$.  If $R = \widehat{S/I}$, then  
\[
\Lambda(R) =\begin{pmatrix}
 0 &  0 & 1 & 0 \\
 0 &  0 & 0 & 0 \\
 0 &  0 & 0 & 1 \\
 0 &  0 & 0 & 1 \\
\end{pmatrix}, \ \text{ but } \ 
\widetilde{\Lambda}(R) = \begin{pmatrix}
 0 & 0 & 0 & 0 \\
 0 & 0 & 0 & 0 \\
 0 & 0 & 0 & 0 \\
 0 & 0 & 0 & 1 \\
\end{pmatrix}.
\]
\end{example}

This example is calculated by utilizing the second computation of \`Alvarez Montaner and Vahidi in Example \ref{LyuNum S-R differ}, as well as a result on Bockstein homomorphisms of local cohomology modules of Singh and Walther \cite{AMV, SinghWaltherBockstein}.

\section*{Acknowledgments}
We send many thanks to Josep \`Alvarez Montaner, Manuel Blickle,
Daniel J.\- Hern\'andez, Mel Hochster, Gennady Lyubeznik, Felipe P\'erez, and Uli Walther
for useful mathematical conversations and suggestions related to this work.
We also thank  Josep \`Alvarez Montaner
for computing several examples that appear in this survey.
The first author also thanks No\'e B\'arcenas Torres, Fernando Galaz Garc\'ia, and  M\'onica Moreno Rocha for organizing
the ``Taller de Vinculaci\'on Matem\'aticos Mexicanos J\'ovenes en el Mundo'' in Guanajuato, M\'exico, where this project was initiated,
as well as the National Council of Science and Technology of Mexico (CONACyT) for its support through Grant $210916.$
{\edit{ We thank the referee for a careful reading of this paper, and for several comments and suggestions that improved it.
}}
\bibliographystyle{skalpha}
\bibliography{CommonBib}

\newcommand{\etalchar}[1]{$^{#1}$}
\def\cprime{$'$} \def\cprime{$'$}
  \def\cfudot#1{\ifmmode\setbox7\hbox{$\accent"5E#1$}\else
  \setbox7\hbox{\accent"5E#1}\penalty 10000\relax\fi\raise 1\ht7
  \hbox{\raise.1ex\hbox to 1\wd7{\hss.\hss}}\penalty 10000 \hskip-1\wd7\penalty
  10000\box7}
\providecommand{\bysame}{\leavevmode\hbox to3em{\hrulefill}\thinspace}
\providecommand{\MR}{\relax\ifhmode\unskip\space\fi MR}
\providecommand{\MRhref}[2]{%
  \href{http://www.ams.org/mathscinet-getitem?mr=#1}{#2}
}
\providecommand{\href}[2]{#2}
\begin{thebibliography}{{\`A}MGLZA03}

\bibitem[{\`A}MV14]{AMV}
{\sc J.~{\`A}lvarez~Montaner and A.~Vahidi}: \emph{Lyubeznik numbers of
  monomial ideals}, Trans. Amer. Math. Soc. \textbf{366} (2014), no.~4,
  1829--1855.

\bibitem[{\`A}M00]{AM-CharCyclesI}
{\sc J.~{\`A}lvarez~Montaner}: \emph{Characteristic cycles of local cohomology
  modules of monomial ideals}, J. Pure Appl. Algebra \textbf{150} (2000),
  no.~1, 1--25. {\sf\scriptsize 1762917 (2001d:13016)}

\bibitem[{\`A}M04a]{AM-CharCyclesII}
{\sc J.~{\`A}lvarez~Montaner}: \emph{Characteristic cycles of local cohomology
  modules of monomial ideals. {II}}, J. Pure Appl. Algebra \textbf{192} (2004),
  no.~1-3, 1--20. {\sf\scriptsize 2067186}

\bibitem[{\`A}M04b]{AM-NumInv}
{\sc J.~{\`A}lvarez~Montaner}: \emph{Some numerical invariants of local rings},
  Proc. Amer. Math. Soc. \textbf{132} (2004), no.~4, 981--986 (electronic).
  {\sf\scriptsize 2045412 (2005c:13021)}

\bibitem[{\`A}M13]{MONICA}
{\sc J.~{\`A}lvarez~Montaner}: \emph{Local cohomology supported on monomial
  ideals}, Monomial ideals, computations and applications, Lecture Notes in
  Mathematics, 2013, pp.~125--198.

\bibitem[{\`A}M14]{AM-SCM}
{\sc J.~{\`A}lvarez~Montaner}: \emph{Lyubeznik table of sequentially
  cohen-macaulay rings}, To Appear in Communications in Algebra (2014).

\bibitem[{\`A}MFR13]{Tutorial}
{\sc J.~{\`A}lvarez~Montaner and O.~Fern{\'a}ndez-Ramos}: \emph{Local
  cohomology using macaulay2}, Monomial ideals, computations and applications,
  Lecture Notes in Mathematics, 2013.

\bibitem[{\`A}MGLZA03]{AM-GL-Z}
{\sc J.~{\`A}lvarez~Montaner, R.~Garc{\'{\i}}a~L{\'o}pez, and
  S.~Zarzuela~Armengou}: \emph{Local cohomology, arrangements of subspaces and
  monomial ideals}, Adv. Math. \textbf{174} (2003), no.~1, 35--56.
  {\sf\scriptsize 1959890 (2004a:13012)}

\bibitem[{\`A}ML06]{AlvarezMontanerLeykin}
{\sc J.~{\`A}lvarez~Montaner and A.~Leykin}: \emph{Computing the support of
  local cohomology modules}, J. Symbolic Comput. \textbf{41} (2006), no.~12,
  1328--1344. {\sf\scriptsize 2271328 (2008a:13023)}

\bibitem[Bj{\"o}79]{BjorkRingsDifferentialOperators}
{\sc J.-E. Bj{\"o}rk}: \emph{Rings of differential operators}, North-Holland
  Mathematical Library, vol.~21, North-Holland Publishing Co., Amsterdam, 1979.
  {\sf\scriptsize 549189 (82g:32013)}

\bibitem[Bli04]{BlickleIntHomology}
{\sc M.~Blickle}: \emph{The intersection homology {$D$}-module in finite
  characteristic}, Math. Ann. \textbf{328} (2004), no.~3, 425--450.
  {\sf\scriptsize 2036330 (2005a:14005)}

\bibitem[Bli07]{BlickleLyuCIS}
{\sc M.~Blickle}: \emph{Lyubeznik's invariants for cohomologically isolated
  singularities}, J. Algebra \textbf{308} (2007), no.~1, 118--123.
  {\sf\scriptsize 2290913 (2008b:13024)}

\bibitem[BB05]{B-B}
{\sc M.~Blickle and R.~Bondu}: \emph{Local cohomology multiplicities in terms
  of \'etale cohomology}, Ann. Inst. Fourier (Grenoble) \textbf{55} (2005),
  no.~7, 2239--2256. {\sf\scriptsize 2207383 (2007d:14009)}

\bibitem[BM10]{Brenner-Monsky}
{\sc H.~Brenner and P.~Monsky}: \emph{Tight closure does not commute with
  localization}, Ann. of Math. (2) \textbf{171} (2010), no.~1, 571--588.
  {\sf\scriptsize 2630050 (2011d:13005)}

\bibitem[BS98]{BrodmannSharpLocalCohomology}
{\sc M.~P. Brodmann and R.~Y. Sharp}: \emph{Local cohomology: an algebraic
  introduction with geometric applications}, Cambridge Studies in Advanced
  Mathematics, vol.~60, Cambridge University Press, Cambridge, 1998.
  {\sf\scriptsize MR1613627 (99h:13020)}

\bibitem[Coh46]{Cohen}
{\sc I.~S. Cohen}: \emph{On the structure and ideal theory of complete local
  rings}, Trans. Amer. Math. Soc. \textbf{59} (1946), 54--106. {\sf\scriptsize
  0016094 (7,509h)}

\bibitem[Eis95]{EisenbudCommutativeAlgebra}
{\sc D.~Eisenbud}: \emph{Commutative algebra}, Graduate Texts in Mathematics,
  vol. 150, Springer-Verlag, New York, 1995, With a view toward algebraic
  geometry. {\sf\scriptsize 1322960 (97a:13001)}

\bibitem[GLS98]{GarciaSabbah}
{\sc R.~Garc{\'{\i}}a~L{{\'o}}pez and C.~Sabbah}: \emph{Topological computation
  of local cohomology multiplicities}, Collect. Math. \textbf{49} (1998),
  no.~2-3, 317--324, Dedicated to the memory of Fernando Serrano.
  {\sf\scriptsize 1677136 (2000a:13029)}

\bibitem[GS]{Macaulay}
{\sc D.~R. Grayson and M.~E. Stillman}: \emph{Macaulay2, a software system for
  research in algebraic geometry}.

\bibitem[Har62]{HartshorneCompleteIntersectionConnectedness}
{\sc R.~Hartshorne}: \emph{Complete intersections and connectedness}, Amer. J.
  Math. \textbf{84} (1962), 497--508. {\sf\scriptsize 0142547 (26 \#116)}

\bibitem[Har66]{HartshorneResidues}
{\sc R.~Hartshorne}: \emph{Residues and duality}, Lecture notes of a seminar on
  the work of A. Grothendieck, given at Harvard 1963/64. With an appendix by P.
  Deligne. Lecture Notes in Mathematics, No. 20, Springer-Verlag, Berlin, 1966.
  {\sf\scriptsize MR0222093 (36 \#5145)}

\bibitem[Har75]{HartshorneOntheDeRhamCohomology}
{\sc R.~Hartshorne}: \emph{On the {D}e {R}ham cohomology of algebraic
  varieties}, Inst. Hautes {\'E}tudes Sci. Publ. Math. (1975), no.~45, 5--99.
  {\sf\scriptsize 0432647 (55 \#5633)}

\bibitem[HS08]{HellusSchenzel}
{\sc M.~Hellus and P.~Schenzel}: \emph{On cohomologically complete
  intersections}, J. Algebra \textbf{320} (2008), no.~10, 3733--3748.
  {\sf\scriptsize 2457720 (2010a:13029)}

\bibitem[HH89]{HochsterHunekeTightClosureAndStrongFRegularity}
{\sc M.~Hochster and C.~Huneke}: \emph{Tight closure and strong
  {$F$}-regularity}, M\'em. Soc. Math. France (N.S.) (1989), no.~38, 119--133,
  Colloque en l'honneur de Pierre Samuel (Orsay, 1987). {\sf\scriptsize
  MR1044348 (91i:13025)}

\bibitem[HH90]{HochsterHunekeTC1}
{\sc M.~Hochster and C.~Huneke}: \emph{Tight closure, invariant theory, and the
  {B}rian\c con-{S}koda theorem}, J. Amer. Math. Soc. \textbf{3} (1990), no.~1,
  31--116. {\sf\scriptsize MR1017784 (91g:13010)}

\bibitem[HH94a]{HochsterHunekeIndecomposable}
{\sc M.~Hochster and C.~Huneke}: \emph{Indecomposable canonical modules and
  connectedness}, Commutative algebra: syzygies, multiplicities, and birational
  algebra (South Hadley, MA, 1992), Contemp. Math., vol. 159, Amer. Math. Soc.,
  Providence, RI, 1994, pp.~197--208. {\sf\scriptsize MR1266184 (95e:13014)}

\bibitem[HH94b]{HochsterHunekeTightClosureOfParameterIdealsAndSplitting}
{\sc M.~Hochster and C.~Huneke}: \emph{Tight closure of parameter ideals and
  splitting in module-finite extensions}, J. Algebraic Geom. \textbf{3} (1994),
  no.~4, 599--670. {\sf\scriptsize MR1297848 (95k:13002)}

\bibitem[HR76]{HochsterRobertsFrobeniusLocalCohomology}
{\sc M.~Hochster and J.~L. Roberts}: \emph{The purity of the {F}robenius and
  local cohomology}, Advances in Math. \textbf{21} (1976), no.~2, 117--172.
  {\sf\scriptsize MR0417172 (54 \#5230)}

\bibitem[HL90]{HunekeLyuVanishing}
{\sc C.~Huneke and G.~Lyubeznik}: \emph{On the vanishing of local cohomology
  modules}, Invent. Math. \textbf{102} (1990), no.~1, 73--93. {\sf\scriptsize
  1069240 (91i:13020)}

\bibitem[Hun96]{HunekeTightClosureBook}
{\sc C.~Huneke}: \emph{Tight closure and its applications}, CBMS Regional
  Conference Series in Mathematics, vol.~88, Published for the Conference Board
  of the Mathematical Sciences, Washington, DC, 1996, With an appendix by
  Melvin Hochster. {\sf\scriptsize MR1377268 (96m:13001)}

\bibitem[HS93]{Huneke-Sharp}
{\sc C.~L. Huneke and R.~Y. Sharp}: \emph{Bass numbers of local cohomology
  modules}, Trans. Amer. Math. Soc. \textbf{339} (1993), no.~2, 765--779.
  {\sf\scriptsize 1124167 (93m:13008)}

\bibitem[ILL{\etalchar{+}}07]{24HoursLocalCohomology}
{\sc S.~B. Iyengar, G.~J. Leuschke, A.~Leykin, C.~Miller, E.~Miller, A.~K.
  Singh, and U.~Walther}: \emph{Twenty-four hours of local cohomology},
  Graduate Studies in Mathematics, vol.~87, American Mathematical Society,
  Providence, RI, 2007. {\sf\scriptsize 2355715 (2009a:13025)}

\bibitem[Kaw02]{KawasakiHighestLyubeznikNumber}
{\sc K.-i. Kawasaki}: \emph{On the highest {L}yubeznik number}, Math. Proc.
  Cambridge Philos. Soc. \textbf{132} (2002), no.~3, 409--417. {\sf\scriptsize
  1891679 (2003b:13026)}

\bibitem[Lyu93]{LyubeznikFinitenessLocalCohomologyModules}
{\sc G.~Lyubeznik}: \emph{Finiteness properties of local cohomology modules (an
  application of {$D$}-modules to commutative algebra)}, Invent. Math.
  \textbf{113} (1993), no.~1, 41--55. {\sf\scriptsize 1223223 (94e:13032)}

\bibitem[Lyu97]{LyubeznikFModulesApplicationsToLocalCohomology}
{\sc G.~Lyubeznik}: \emph{{$F$}-modules: applications to local cohomology and
  {$D$}-modules in characteristic {$p>0$}}, J. Reine Angew. Math. \textbf{491}
  (1997), 65--130. {\sf\scriptsize MR1476089 (99c:13005)}

\bibitem[Lyu00a]{LyubeznikFreeChar}
{\sc G.~Lyubeznik}: \emph{Finiteness properties of local cohomology modules: a
  characteristic-free approach}, J. Pure Appl. Algebra \textbf{151} (2000),
  no.~1, 43--50. {\sf\scriptsize 1770642 (2001g:13038)}

\bibitem[Lyu00b]{LyuUMC}
{\sc G.~Lyubeznik}: \emph{Finiteness properties of local cohomology modules for
  regular local rings of mixed characteristic: the unramified case}, Comm.
  Algebra \textbf{28} (2000), no.~12, 5867--5882, Special issue in honor of
  Robin Hartshorne. {\sf\scriptsize 1808608 (2002b:13028)}

\bibitem[Lyu00c]{LyubeznikInjectiveDim}
{\sc G.~Lyubeznik}: \emph{Injective dimension of {$D$}-modules: a
  characteristic-free approach}, J. Pure Appl. Algebra \textbf{149} (2000),
  no.~2, 205--212. {\sf\scriptsize 1757731 (2001g:13029)}

\bibitem[Lyu02]{LyuSurveyLC}
{\sc G.~Lyubeznik}: \emph{A partial survey of local cohomology}, Local
  cohomology and its applications ({G}uanajuato, 1999), Lecture Notes in Pure
  and Appl. Math., vol. 226, Dekker, New York, 2002, pp.~121--154.
  {\sf\scriptsize 1888197 (2003b:14006)}

\bibitem[Lyu06a]{LyubeznikSomeLocalCohomologyInvariants}
{\sc G.~Lyubeznik}: \emph{On some local cohomology invariants of local rings},
  Math. Z. \textbf{254} (2006), no.~3, 627--640. {\sf\scriptsize 2244370
  (2007g:13024)}

\bibitem[Lyu06b]{LyubeznikCD}
{\sc G.~Lyubeznik}: \emph{On the vanishing of local cohomology in
  characteristic {$p>0$}}, Compos. Math. \textbf{142} (2006), no.~1, 207--221.
  {\sf\scriptsize 2197409 (2007b:13029)}

\bibitem[MZ13]{EulerianDModules}
{\sc L.~Ma and W.~Zhang}: \emph{Eulerian graded {$D$}-modules}, to appear in
  Mathematical Research Letters (2013), arXiv:1210.8402.

\bibitem[Mat89]{MatsumuraCommutativeRingTheory}
{\sc H.~Matsumura}: \emph{Commutative ring theory}, second ed., Cambridge
  Studies in Advanced Mathematics, vol.~8, Cambridge University Press,
  Cambridge, 1989, Translated from the Japanese by M. Reid. {\sf\scriptsize
  MR1011461 (90i:13001)}

\bibitem[Mil00]{MillerE}
{\sc E.~Miller}: \emph{The {A}lexander duality functors and local duality with
  monomial support}, J. Algebra \textbf{231} (2000), no.~1, 180--234.
  {\sf\scriptsize 1779598 (2001k:13028)}

\bibitem[MS05]{MillerSturmfels}
{\sc E.~Miller and B.~Sturmfels}: \emph{Combinatorial commutative algebra},
  Graduate Texts in Mathematics, vol. 227, Springer-Verlag, New York, 2005.
  {\sf\scriptsize 2110098 (2006d:13001)}

\bibitem[Mus00]{MustataSRrings}
{\sc M.~Musta{\c{t}}{\v{a}}}: \emph{Local cohomology at monomial ideals}, J.
  Symbolic Comput. \textbf{29} (2000), no.~4-5, 709--720, Symbolic computation
  in algebra, analysis, and geometry (Berkeley, CA, 1998). {\sf\scriptsize
  1769662 (2001i:13020)}

\bibitem[NB13]{Nunez}
{\sc L.~N\'u{\~n}ez-Betancourt}: \emph{Local cohomology modules of polynomial
  or power series rings over rings of small dimension}, Illinois J. Math.
  \textbf{57} (2013), no.~1, 279--294.

\bibitem[NBP13]{NunezPerezFJacobian}
{\sc L.~N{\'u}{\~n}ez-Betancourt and F.~P{\'e}rez}: \emph{{$F$}-jumping and
  {$F$}-{J}acobian ideals for hypersurfaces}, Preprint (2013), arXiv:1302.3327.

\bibitem[NBW13]{NunezBWittLyuNumMixed}
{\sc L.~N{\'u}{\~n}ez-Betancourt and E.~E. Witt}: \emph{Lyubeznik numbers in
  mixed characteristic}, Math. Res. Lett. \textbf{20} (2013), no.~6,
  1125--1143.

\bibitem[NBW14]{NuWiEqual}
{\sc L.~N{\'u}{\~n}ez-Betancourt and E.~E. Witt}: \emph{Generalized {L}yubeznik
  numbers}, to appear in Nagoya Mathematical Journal (2014), arXiv:1208.5500.

\bibitem[Ogu73]{OgusLocalCohomologicalDimension}
{\sc A.~Ogus}: \emph{Local cohomological dimension of algebraic varieties},
  Ann. of Math. (2) \textbf{98} (1973), 327--365. {\sf\scriptsize 0506248 (58
  \#22059)}

\bibitem[Rei76]{Reisner}
{\sc G.~A. Reisner}: \emph{Cohen-{M}acaulay quotients of polynomial rings},
  Advances in Math. \textbf{21} (1976), no.~1, 30--49. {\sf\scriptsize 0407036
  (53 \#10819)}

\bibitem[Sch11]{Schenzel2011}
{\sc P.~Schenzel}: \emph{On the structure of the endomorphism ring of a certain
  local cohomology module}, J. Algebra \textbf{344} (2011), 229--245.
  {\sf\scriptsize 2831938 (2012j:13026)}

\bibitem[SW11]{SinghWaltherBockstein}
{\sc A.~K. Singh and U.~Walther}: \emph{Bockstein homomorphisms in local
  cohomology}, J. Reine Angew. Math. \textbf{655} (2011), 147--164.
  {\sf\scriptsize 2806109}

\bibitem[Swi14]{Switala}
{\sc N.~Switala}: \emph{Lyubeznik numbers for nonsingular projective
  varieties}, Preprint (2014), arXiv:1403.5776.

\bibitem[Vas98]{Vassilev}
{\sc J.~C. Vassilev}: \emph{Test ideals in quotients of {$F$}-finite regular
  local rings}, Trans. Amer. Math. Soc. \textbf{350} (1998), no.~10,
  4041--4051. {\sf\scriptsize 1458336 (98m:13009)}

\bibitem[Wal99]{WaltherAlgorithm}
{\sc U.~Walther}: \emph{Algorithmic computation of local cohomology modules and
  the local cohomological dimension of algebraic varieties}, J. Pure Appl.
  Algebra \textbf{139} (1999), no.~1-3, 303--321, Effective methods in
  algebraic geometry (Saint-Malo, 1998). {\sf\scriptsize 1700548 (2000h:13012)}

\bibitem[Wal01]{WaltherLyubeznikNumbers}
{\sc U.~Walther}: \emph{On the {L}yubeznik numbers of a local ring}, Proc.
  Amer. Math. Soc. \textbf{129} (2001), no.~6, 1631--1634 (electronic).
  {\sf\scriptsize 1814090 (2001m:13026)}

\bibitem[Wal05]{WaltherBS}
{\sc U.~Walther}: \emph{Bernstein-{S}ato polynomial versus cohomology of the
  {M}ilnor fiber for generic hyperplane arrangements}, Compos. Math.
  \textbf{141} (2005), no.~1, 121--145. {\sf\scriptsize 2099772 (2005k:32030)}

\bibitem[Yan01a]{Yanagawa}
{\sc K.~Yanagawa}: \emph{Bass numbers of local cohomology modules with supports
  in monomial ideals}, Math. Proc. Cambridge Philos. Soc. \textbf{131} (2001),
  no.~1, 45--60. {\sf\scriptsize 1833073 (2002c:13037)}

\bibitem[Yan01b]{YanagawaSheaves}
{\sc K.~Yanagawa}: \emph{Sheaves on finite posets and modules over normal
  semigroup rings}, J. Pure Appl. Algebra \textbf{161} (2001), no.~3, 341--366.
  {\sf\scriptsize 1836965 (2002k:13033)}

\bibitem[Zha07]{ZhanghighestLyubeznikNumber}
{\sc W.~Zhang}: \emph{On the highest {L}yubeznik number of a local ring},
  Compos. Math. \textbf{143} (2007), no.~1, 82--88. {\sf\scriptsize 2295196
  (2008g:13029)}

\bibitem[Zha11]{ZhangLyubeznikNumbersProjSchemes}
{\sc W.~Zhang}: \emph{Lyubeznik numbers of projective schemes}, Adv. Math.
  \textbf{228} (2011), no.~1, 575--616. {\sf\scriptsize 2822240 (2012j:13027)}

\bibitem[Zha12]{YiZhang}
{\sc Y.~Zhang}: \emph{Graded {$F$}-modules and local cohomology}, Bull. Lond.
  Math. Soc. \textbf{44} (2012), no.~4, 758--762. {\sf\scriptsize 2967243}

\bibitem[Zho98]{ZhouHigherDer}
{\sc C.~Zhou}: \emph{Higher derivations and local cohomology modules}, J.
  Algebra \textbf{201} (1998), no.~2, 363--372. {\sf\scriptsize 1612378
  (99c:13038)}

\end{thebibliography}
\end{document}